\documentclass[12pt]{article}
\usepackage{amsmath}
\usepackage{graphicx,psfrag,epsf}
\usepackage{enumerate}
\usepackage{natbib}
\usepackage{amssymb}        
\usepackage{amsthm}
\usepackage{multirow, subfigure}
\usepackage{color}

\newcommand{\pr}{\text{pr }}
\newcommand{\E}{\text{E}}

\newtheorem{thm}{Theorem}[section]
\newtheorem{lemma}{Lemma}[section]

\begin{document}

\def\spacingset#1{\renewcommand{\baselinestretch}%
{#1}\small\normalsize} \spacingset{2}


\begin{center}
{\LARGE {\bf A residual-based bootstrap for functional autoregressions}} \vspace*{5 mm}\\

J{\"u}rgen Franke$^1$, Euna Gesare Nyarige$^2$, \\
$^1$ Department of Mathematics, Technische Universit{\"a}t Kaiserslautern\\
$^2$ Department of Mathematics and Statistics, Machakos University
\vspace*{1 cm}\\
\end{center}

\begin{abstract}
We consider the residual-based or naive bootstrap for functional autoregressions of order 1 and prove that it is asymptotically valid for, e.g., the sample mean and for  empirical covariance operator estimates. As a crucial auxiliary result, we also show that the empirical distribution of the centered sample innovations converges to the distribution of the innovations with respect to the Mallows metric.
\end{abstract}

\noindent%
\vfill

\spacingset{1.45} 

\section{Introduction}

The seminal work of \citep{Bosq} has initiated a lot of research on the theory, computational aspects and applications of functional data analysis. The recent monograph of \citep{Horvath} and, with a focus on functional time series, the review article of \citep{Kokoszka12} give an overview over the field of research. In this paper, we consider a time series $X_t \in \mathcal{H}, -\infty < t < \infty,$ with values in a Hilbert space $\mathcal{H}$, e.g. curves in a function space like 
L$^2[0,1]$. In particular, we are interested in functional autoregressions, also known as autoregressive Hilbertian models (ARH). As is well known, a functional autoregressive process of order $p$ or FAR($p$)-process can be easily be written as a FAR(1)-process by an appropriate change of state vector and Hilbert space. Therefore, it essentially suffices to consider the case of order 1, where
\begin{equation} \label{FAR1}
X_{t+1}=\Psi(X_t) +\epsilon_{t+1} .
\end{equation}
Here, $\Psi: \mathcal{H} \to \mathcal{H}$ is a linear operator, and $\epsilon_t \in \mathcal{H}$ are independent, identically distributed (i.i.d.) innovations. Recently, several new statistical methods for data generated by (\ref{FAR1}) have been proposed, in particular regarding tests and forecasts. \citep{Kokoszka} have investigated a test of the hypothesis $\Psi=0$, i.e. of independence of the data. \citep{Gabrys} consider a related problem, a test of independence for general functional time series. \citep{HHK} propose a CUSUM test for a sudden change in the dependence structure of the data, i.e. for the presence of a point in time where the value of $\Psi$ changes, which has been applied to neurophysiological data by \citep{Franke18}. Other papers concentrate on the task of forecasting the data. \citep{Didericksen} present an empirical study of forecasting $X_{t+1}$ by $\hat{\Psi}(X_t)$ where $\hat{\Psi}$ denotes some estimate of $\Psi$. \citep{Kargin} develop an appropriate theory for a particular kind of estimate $\hat{\Psi}$. Also, forecasting on the basis of FAR(1) models has been used in a lot of applications partly discussed below in the context of the bootstrap.\\

Asymptotics for the distribution of estimates of the autoregressive operator $\Psi$ is involved, as pointed out by \citep{Mas}, and as, additionally, it frequently provides decent approximations only for large sample sizes, a lot of applied papers use resampling techniques to derive critical values for tests or prediction intervals for forecasts (compare, e.g., \citep{Shang} for an overview). The theory for bootstrapping functional data, which provides guidelines under which circumstances bootstrap approximations are valid, is, however, still rather incomplete. E.g., only recently \citep{Paparoditis} show that bootstrap methods work for testing the equality of means and covariance operators in $K$ samples of independent functional data.

We are, in particular, interested in the residual-based bootstrap where resampling is done on the basis of the centered sample residuals $\hat{\epsilon}_t = X_t - \hat{\Psi}(X_{t-1})$. This kind of bootstrap is quite common in the context of scalar autoregressive and ARMA models (compare \citep{Kreiss}) and forms the starting point for the widely applicable autoregressive sieve bootstrap (compare \citep{KreissPP}). 

This kind of bootstrap has been investigated in the analogous, but, from the viewpoint of theory, considerably simpler regression situation. \citep{GMMC} discuss the linear functional regression model $Y_t = \Psi(X_t) + \epsilon_t$, where $Y_t$ is scalar and $\Psi: \mathcal{H} \to \mathbb{R}$ is a linear functional. Treating $X_t$ as fixed which is common in the regression context, they prove that the residual-based bootstrap and, for heteroscedastic residuals $\epsilon_t$, the wild bootstrap works. In the same model, \citep{GMGRMCGP} apply the pairwise bootstrap and the wild bootstrap to a test of the hypothesis $\Psi=0$. \citep{Ferraty10} consider the functional regression model with general, not necessarily linear operators $\Psi$ and prove that the residual-based and the wild bootstrap works for nonparametric kernel estimates of $\Psi$. \citep{Ferraty12} extend those results to the case where the response variable is also of functional nature, e.g. $Y_t \in \mathcal{H}$. \citep{Zhu} and \citep{Rana} discuss the analogous situation for nonparametric functional autoregressions, considering the regression bootstrap and the wild bootstrap respectively (compare \citep{Franke02} for these concepts, their advantages and drawbacks in the scalar case), but not the residual-based bootstrap. 

Bootstrap techniques are also quite popular in approximating the distribution of statistics from  functional time series data. \citep{Horvath} use in their section 14.1 the residual-based bootstrap for evaluating the performance of a test for a change in the autoregressive operator of a FAR(1)-process. \citep{APCVF} consider the nonparametric FAR($d$)-model $X_t = m(X_{t-1},\ldots,X_{t-d}) + \epsilon_t$, estimate the autoregression operator $m(.)$ nonparametrically by kernel and local linear estimates and apply the residual-based bootstrap to get prediction intervals. \citep{Romo} discuss the residual-based bootstrap for the integrated FAR(1)-model, i.e. for the special case $\Psi = Id_{\mathcal{H}}$, the Hilbert space identity. They derive bootstrap approximations of critical bounds for unit root tests where, under the hypothesis, $\Psi$ is known. \citep{Castro} investigate among other bootstrap techniques a variant of the residual-based bootstrap in forecasting applications. They start from the centered sample residuals, but do not resample directly from their empirical distribution. They first consider a finite principal component decomposition of the sample residuals and, then, resample the coefficients of this decomposition separately. In a similar spirit, \citep{Hyndman} assume from the start that the time series has a finite Karhunen-Lo{\`e}ve expansion which allows to reduce the functional time series to the finite-dimensional time series of the coefficients. They derive bootstrap  prediction intervals based on bootstrap confidence intervals for the scalar coefficient time series. All these papers focus on simulations and applications and do not consider the accompanying theory. This gap is filled for the stationary bootstrap, which is a variant of the well-known block bootstrap with random block lengths, in an early paper of \citep{Politis}. They consider general Hilbert space valued times series and prove, based on a central limit theorem for triangular arrays of such data, that this bootstrap provides valid approximations for the asymptotic distribution of certain statistics.

Based on the thesis \citep{Nyarige}, we show in this paper that the residual-based bootstrap is applicable to FAR(1)-processes. The theory has direct practical implications as, e.g., the necessary centering of the lag-1 autocovariance operator in the bootstrap world is different from what one would naively expect due to the particular nature of the estimate of $\Psi$. For the proof, we cannot use the approach of \citep{GMMC} for the residual-based bootstrap in regression and of \citep{Politis} for the stationary bootstrap who, for the bootstrap data, both mimic the proof of asymptotic normality of the corresponding functions of the real data. We have to use different methods which are similar to the scalar situation presented by \citep{Franke}; more details will be given in section \ref{secboot}.\\

In section \ref{secmod} we describe the details of our model including the relevant assumptions, and we introduce some estimates from the literature which we need later on. 

In section \ref{secres} we present the crucial result that the empirical distribution of the centered sample innovations converges to the distribution of the innovations. 

In section \ref{secboot} we give the details for the residual-based bootstrap and, as an illustration, state that it works for estimates of the mean and of the first two covariance operators of the data.

Finally, technical results and proofs are given in the appendix.\\

\section{The Model and the Estimates} \label{secmod}

In this section, we mainly collect some properties of our model and some estimates which are standard in the literature on functional autoregressions and which we need later on. This also serves to introduce notation. 

Let $\mathcal{H}$ be separable Hilbert space with scalar product $\langle ., . \rangle$ and norm $\Vert . \Vert$. As a norm for bounded linear operators from $\mathcal{H}$ to $\mathcal{H}$ like $\Psi$ we use 
$$
\Vert\Psi\Vert_\mathcal{L} = \sup \{\Vert \Psi(x) \Vert; \Vert x \Vert = 1 \}.
$$
A sufficient condition for the existence of a stationary solution of (\ref{FAR1}) is $\Vert\Psi\Vert_\mathcal{L} < 1$ (compare \citep{Bosq}, section 3.2). We call a linear operator 
$\Psi$ \textit{compact} if for two orthonormal bases $v_j, j \ge 1,$ and $u_j, j \ge 1,$ of $\mathcal{H}$ and a sequence of real numbers $\gamma_j, j \ge 1,$ converging to 0,
$$
\Psi(x) = \sum_{j=1}^\infty \gamma_j \langle x,v_j \rangle u_j , \quad x \in \mathcal{H}.
$$
$\Psi$ is, in particular, a \textit{Hilbert-Schmidt operator} if $\Vert \Psi \Vert_\mathcal{S}^2 = \sum_{j=1}^\infty \Vert \Psi(v_j) \Vert^2 = \sum_{j=1}^\infty \gamma_j^2 < \infty$. The Hilbert-Schmidt norm $\Vert \Psi \Vert_\mathcal{S}$ is an upper bound for $\Vert \Psi \Vert_\mathcal{L}$. The Hilbert-Schmidt operators $A, B: \mathcal{H} \to \mathcal{H}$ form a Hilbert space themselves with a scalar product given by
$$
\langle A, B \rangle_\mathcal{S} = \sum_{j=1}^\infty \langle A(u_j), B(u_j)\rangle
$$
for an arbitrary orthonormal basis $u_1, u_2, \ldots$ of $\mathcal{H}$ (compare \citep{Horvath}, section 2.1).\\

For the definition of covariance operators, it is convenient to introduce the Kronecker product $y \otimes z$ of $y, z \in \mathcal{H}$ which is a linear operator defined by
$$
(y \otimes z)(x) = \langle y , x \rangle z, \quad x \in \mathcal{H}.
$$
For later reference, we state two rules of calculation which we use repeatedly and which follow immediately from the definition
\begin{equation} \label{kronecker}
z \otimes y = (y \otimes z)^T, \quad \quad A(y) \otimes B(z) = B (y \otimes z) A^T, \quad \quad y, z \in \mathcal{H},
\end{equation}
where $A, B$ are two linear operators on $\mathcal{H}$ and here and the following $A^T$ denotes the adjoint of the linear operator $A$ which is characterized by $\langle A(y),z\rangle = \langle y,A^T(z)\rangle$ for all $y, z \in \mathcal{H}$.\\

We assume throughout the paper that the data $X_0, \ldots, X_n$ are part of a stationary functional autoregression (\ref{FAR1}) with mean $\E X_t =0$. Correspondingly, the covariance operator and the lag 1-autocovariance operator are given by $\Gamma = \E X_t \otimes X_t$ and $C = \E X_t \otimes X_{t+1}$. Furthermore, we always assume that 0 is not an eigenvalue of $\Gamma$. Then, all eigenvalues $\lambda_1 \ge \lambda_2 \ge \ldots$ of $\Gamma$ are positive. Let $v_1, v_2, \ldots$ denote the corresponding orthonormal eigenvectors in $\mathcal{H}$.

$\Gamma, C$ are related to the autoregressive operator $\Psi$ by the analogue to the scalar Yule-Walker equation 
\begin{equation} \label{YW}
\Psi \Gamma = C
\end{equation}

The mean $\E X_t$ is estimated as usual by the sample mean
$$
\bar{X}_n = \frac{1}{n} \sum_{j=0}^{n-1} X_j .
$$
As estimates of $\Gamma, C$ we follow \citep{Horvath} and use the simplified sample versions
$$
\hat{\Gamma}_n = \frac{1}{n} \sum_{j=0}^{n-1} X_j \otimes X_j, \quad \hat{C}_n = \frac{1}{n}  \sum_{j=0}^{n-1} X_j \otimes X_{j+1}.
$$
We use the last observation $X_n$ only in estimating $C$ to streamline notation later on. Due to the same reason, we do not center the $X_j$ around $\bar{X}_n$ in the definitions of $\hat{\Gamma}_n, \hat{C}_n$. Under our assumption $\E X_t = 0$, this has an asymptotically neglible effect. All results remain true in the general case $\E X_t = \mu \in \mathcal{H}$ but then we of course have to center the data around 0 in calculating the covariance estimates.

$\hat{\lambda}_j, \hat{v}_j$ denote the eigenvalues and eigenvectors of $\hat{\Gamma}_n$. Solving the Yule-Walker equation (\ref{YW}) is an ill-conditioned problem as $\Gamma^{-1}$ is not a bounded linear operator defined on the whole space $\mathcal{H}$. Therefore, $\hat{\Gamma}_n^{-1}$ has to be regularized. We use the popular approach via a finite principal component expansion, compare \citep{Bosq}, \citep{Horvath}, and consider
$$
\hat{\Gamma}^\dagger_n = \sum_{j=1}^{k_n} \frac{1}{\hat{\lambda}_j} \hat{v}_j\otimes \hat{v}_j,
$$
where $k_n \to \infty$ slowly for $n \to \infty$ to get a consistent estimate of $\Psi$. Note that $\hat{\lambda}_j^{-1}$ is an eigenvector of $\hat{\Gamma}_n^{-1}$, and $\hat{v}_j\otimes \hat{v}_j(x)$ is the orthogonal projection of $x$ onto the span of the eigenvector $\hat{v}_j$. Then, we get as an estimate of $\Psi$
$$
\hat{\Psi}_n = \hat{C}_n \hat{\Gamma}^\dagger_n .
$$

\section{Approximation of the innovation distribution by the empirical measure of sample residuals} \label{secres}

The basis for residual-based bootstrapping in scalar regression and autoregression models is the approximability of the innovations by the bootstrap innovations where the latter are drawn from the centered sample residuals. This is stated in the following theorem in terms of the Mallows metric $d_2$ which is discussed in detail by \citep{Bickel}. For two distributions $F, G$ on $\mathcal{H}$, it is defined by
$$
d_2^2(F,G) = \inf_{X, Y} \ \E \Vert X-Y \Vert^2 ,
$$
where the infimum is taken over all $\mathcal{H}$-valued random variables $X$ and $Y$ with marginal distributions $F$ resp. $G$. By Lemma 8.1. of \citep{Bickel} the infimum is attained.

By $F, \hat{F}_n$, we denote the distribution of $\epsilon_t$ respectively the empirical distribution of the centered sample residuals $\tilde{\epsilon}_1,\ldots,\tilde{\epsilon}_n$ with 
\begin{equation} \label{empres}
\tilde{\epsilon}_j=\hat{\epsilon}_j-\dfrac{1}{n}\sum_{k=1}^n\hat{\epsilon}_k, \quad \hat{\epsilon}_j=X_j-\hat{\Psi}_n\left(X_{j-1}\right), \quad j=1,\ldots,n.
\end{equation}

\begin{thm}
Let $X_0,\ldots,X_n$ be a sample from a stationary FAR(1) process satisfying\\
i) $\{\epsilon_t\}$ i.i.d., $\E\epsilon_t=0, \, \E\left\Vert \epsilon_t \right\Vert^4 < \infty$, \\
ii) $\Psi$ is a Hilbert-Schmidt operator with $\left\Vert \Psi \right\Vert_{\mathcal{L}} <1$, \\
iii) the eigenvalues $\lambda_1>\lambda_2>\ldots$ of $\Gamma$ are all positive and have multiplicity 1.\\
Then,
\begin{equation*}
d_2\left(\hat{F}_n,F\right)\underset{p}\rightarrow 0 \quad \text{for } n\to \infty,
\end{equation*}
if $k_n \to \infty$ and, with $a_1=\lambda_1-\lambda_2,\; a_j = \min(\lambda_{j-1}-\lambda_j, \lambda_j - \lambda_{j+1}), \; j \ge 2$,  
\begin{equation} \label{ratekn}
\frac{k_n}{n} \sum_{j=1}^{k_n} \frac{1}{a_j^2} \to 0 \; \text{ for }\; n \to \infty \quad \text{ and } \quad \frac{1}{\lambda_{k_n}} = O\left(\frac{n^{1/4}}{(\log n)^{\beta}}\right) \; \text{ for some }\; \beta > \frac{1}{2}.
\end{equation}
\label{theorem2}
\end{thm}

A fourth moment condition like i) is not unexpected, as $\hat{\Psi}_n$ depends on $\hat{\Gamma}_n, \hat{C}_n$ which are quadratic in the data and which we want to be $\sqrt{n}$-consistent estimates. Condition ii) may be relaxed to $\left\Vert \Psi^{j_0} \right\Vert_{\mathcal{L}} <1$ for some $j_0 \ge 1$ as in the work of \citep{Bosq}; we prefer the somewhat stronger assumption to simplify the proofs. The positivity of the eigenvalues in iii) is necessary to exclude singular cases. Assuming dimension 1 of all eigenspaces is standard in the literature on functional autoregressions to circumvent the notational problems with the nonuniqueness of eigenvectors generating a particular eigenspace, but it is not essential for the validity of the results.

The following lemma illustrates the meaning of the rate condition (\ref{ratekn}) for two particular examples where we impose lower bounds on $\kappa_j = \lambda_j-\lambda_{j+1}$ which is related to the rate of decrease of the eigenvalues. If $\kappa_j$ is allowed to decrease exponentially fast, then $k_n$ may increase at most logarithmically in $n$. If $\kappa_j$ may converge to 0 only with a polynomial rate in $j^{-1}$ then $k_n$ may increase faster like $n^c$ for appropriate $c>0$. These kinds of relationship between $k_n$ and the rate of decrease of the eigenvalues $\lambda_j$ is quite plausible regarding the character of $k_n$ as a regularization parameter. In similar situations, \citep{Guillas} found the same kind of rate conditions in his study of the convergence rate of $\hat{\Psi}_n$. 

\begin{lemma} \label{logratekn1}
a) Let $\lambda_j - \lambda_{j+1} \ge b a^j, \ j= 1, 2, \ldots$ for some $0 < a < 1, b > 0$. Then, 
(\ref{ratekn}) is satisfied for $n, k_n \to \infty$ if, for all large enough $n$, 
\begin{equation*}
k_n \le \left(\dfrac{1}{4\log \frac{1}{a}} - \delta \right) \log n \quad \text{ for some }\ \delta > 0.
\end{equation*}
b) Let $\lambda_j - \lambda_{j+1} \ge b j^{-a}, \ j= 1, 2, \ldots$ for some $a > 1, b > 0$. Then, 
(\ref{ratekn}) is satisfied for $n, k_n \to \infty$ if 
\begin{equation*}
k_n = O\left( n^{\frac{1}{4a}-\delta}\right)  \; \text{ for some }\ \delta > 0.
\end{equation*}
\end{lemma}
\begin{proof}
a) From the condition of the lemma, we immediately have $\lambda_j \ge  a_j \ge b a^j$. Using the formula for geometric sums, 
$$
\frac{k_n}{n} \sum_{j=1}^{k_n} \frac{1}{a_j^2} \le \frac{k_n}{n b^2 a^{2 k_n}} \sum_{j=1}^{k_n} a^{2(k_n-j)} \le \frac{1}{b^2 (1-a^2)} \ \frac{k_n}{n a^{2 k_n}} \to 0
$$
as $\log n - 2 k_n \log \frac{1}{a} - \log k_n \ge \frac{1}{2} \log n - \log \log n - \log c \to \infty$. Moreover we have
$$
\frac{1}{\lambda_{k_n}} \le \frac{1}{b a^{k_n}} \le \frac{1}{b} \frac{n^{1/4}}{(\log n)^{\beta}}
$$
for large enough $n$, as, for some $\delta > 0$ and all $\beta > 0$, again for large enough $n$,
$$
k_n \log \frac{1}{a} \le (c \log \frac{1}{a}) \log n \le (\frac{1}{4} - \delta) \log n \le \frac{1}{4} \log n - \beta \log \log n .
$$
b) The proof proceeds in a similar manner as for part a), using $\lambda_j^{-1} \le a_j^{-1} \le \frac{1}{b} j^a$ and
$$
\sum_{j=1}^{k_n} j^{2a} \le \sum_{j=1}^{k_n} \int_j^{j+1} u^{2a} du = \int_1^{k_n+1} u^{2a} du \le \frac{1}{2a+1} (k_n+1)^{2a+1} .
$$
\end{proof}

\section{The residual-based bootstrap} \label{secboot}

We start with a sample $X_0, \ldots, X_n$ from a stationary functional autoregression (\ref{FAR1}). The basic idea of the bootstrap is to replace the data $X_t$ by pseudodata $X_t^*$, calculated from the given sample, with two features:\\
i) The distribution of certain functions $T(X_0, \ldots, X_n)$ of the data can be approximated by the conditional distribution of the corresponding functions $T(X_0^*, \ldots, X_n^*)$ of the pseudodata given $X_0, \ldots, X_n$.\\
ii) The conditional distribution of $T(X_0^*, \ldots, X_n^*)$ given $X_0, \ldots, X_n$ is known such that distributional characteristics like moments or quantiles can be numerically calculated by Monte Carlo simulation.\\

In this section, we generalize the well-known residual-based bootstrap for scalar ARMA-processes, compare, e.g. \citep{Kreiss}, to the functional setting. Let $\tilde{\epsilon}_1, \ldots, \tilde{\epsilon}_n$ be the centered sample residuals given by (\ref{empres}), and let $\hat{F}_n$ be their empirical distribution function. The procedure for generating the pseudodata $X_t^*$ is the following:

\textbf{1)} Draw bootstrap innovations $\epsilon_t^*, t=1, \ldots, n$, purely randomly from the centered sample residuals:
$$
\pr^*(\epsilon_t^* = \tilde{\epsilon}_k) = \frac{1}{n}, \quad t, k = 1, \ldots, n,
$$
such that the $\epsilon_t^*$ are i.i.d. with distribution $\hat{F}_n$ conditional on the original data. Here and in the following, we write $\pr^*, \E^*$ for conditional probabilities and expectations given $X_0, \ldots, X_n$.

\textbf{2)} We generate the bootstrap data $X_t^*, t=1, \ldots, n$, recursively by
$$
X_t^* = \hat{\Psi}_n(X_{t-1}^*) + \epsilon_t^*, \quad t= 1, \ldots, n,
$$ 
for some suitable initial value $X_0^*$.

If $n$ is large, the choice of $X_0^*$ is of minor importance due to the exponentially decreasing memory of our stationary FAR(1)-process. This follows from its representation as an infinite moving average process (e.g. Theorem 13.1 of \citep{Horvath}) together with $\Vert \Psi^j\Vert_\mathcal{L} \le \Vert\Psi \Vert_\mathcal{L}^j$ and $\Vert\Psi\Vert_\mathcal{L} < 1$. Popular choices are $X_0^* = X_0$, which are used in the simulations of \citep{Nyarige}, or $X_0^* = \E X_0 = 0$. \\

Let us remark that the theory of the residual bootstrap has already been studied for the quite similar functional linear regression model $Y_j = \Psi(X_j) + \epsilon_j$ with real-valued $Y_j, \epsilon_j$ and functional regressors $X_j$ by \citep{GMMC}. Note that the situation there is much simpler, not only due to the lack of dependence, but equally due to the fact that, by construction, $X_j^* = X_j$. Therefore, the regressors $X_j^*$ in the bootstrap world trivially satisfy exactly the same assumptions as the real regressors $X_j$ which is quite useful in showing that the same kind of asymptotics holds for functions of the real resp. the bootstrap data. In particular, the critical covariance operator estimate $\hat{\Gamma}_n$, for which we need a regularized inverse, and its eigenvalues and eigenfunctions are the same for the real and the bootstrap data, i.e. Theorem \ref{bootcov} below is trivially satisfied in the regression context. Obviously, for functional autoregressions, those assertions do not hold, and we cannot use the proof of validity of the bootstrap for the regression case at all, but have to use quite different arguments.\\

The regression and wild bootstrap, considered by \citep{Zhu} respectively \citep{Rana} for nonparametric functional autoregressions, also use $X_j^* = X_j$, i.e. they do not mimic the whole time series in the bootstrap world but only the local predictor relationship. So, for proofs, they can rely on the same kind of simpler methods as in the case of regression with independent data.\\

\subsection{Bootstrapping the sample mean}

In this subsection we investigate the sample mean and its analogue in the bootstrap world
$$
\bar{X}_n = \frac{1}{n} \sum_{t=0}^{n-1} X_t, \quad \quad \bar{X}_n^* = \frac{1}{n} \sum_{t=0}^{n-1} X_t^*.
$$
Note that $\E \bar{X}_n = 0$. In the proof, we show that for the bootstrap analogue $\E^* \bar{X}_n^* = 0$ also holds. Therefore, we have to compare the distributions of $\bar{X}_n$ and $\bar{X}_n^*$ without additional centering. In the next theorem and in the following, we use a common convention and write $d_2(X,Y)$ for the Mallows distance $d_2(F,G)$ between the marginal distributions $F, G$ of the random variables $X$ resp. $Y$.\\

\begin{thm}\label{bootmean}
Under the assumptions of Theorem \ref{theorem2} and if $k_n$ satisfies additionally
\begin{equation} \label{ratekn2}
\frac{1}{\lambda_{k_n}} \sum_{j=1}^{k_n} \frac{1}{a_j} = O\left(\frac{n^{1/4}}{(\log n)^{\beta}}\right) \; \text{ for some }\; \beta > 1.
\end{equation}
 we have for $n \to \infty$
\begin{equation*}
n d_2^2\left(\bar{X}_n, \bar{X}_n^* \right) \underset{p}\rightarrow 0
\end{equation*}
\end{thm}

The following lemma provides two examples of a sufficient rate condition for $k_n$ depending on the rate of decrease of $\lambda_j - \lambda_{j+1}$, It is proven in the same manner as Lemma \ref{logratekn1}.

\begin{lemma} \label{logratekn2}
a) Let $\lambda_j - \lambda_{j+1} \ge b a^j, \ j= 1, 2, \ldots$ for some $0 < a < 1, b > 0$. Then, (\ref{ratekn}) and 
(\ref{ratekn2}) are satisfied for $n, k_n \to \infty$ if, for all large enough $n$, 
\begin{equation*}
k_n \le \left(\dfrac{1}{8\log \frac{1}{a}} - \delta \right) \log n \quad \text{ for some }\ \delta > 0.
\end{equation*}
b) Let $\lambda_j - \lambda_{j+1} \ge b j^{-a}, \ j= 1, 2, \ldots$ for some $a > 1, b > 0$. Then, (\ref{ratekn}) and 
(\ref{ratekn2}) are satisfied for $n, k_n \to \infty$ if 
\begin{equation*}
k_n = O\left( n^{\frac{1}{4(2a+1)}-\delta}\right)  \; \text{ for some }\ \delta > 0.
\end{equation*}
\end{lemma}

\subsection{Bootstrapping the covariance operators}

In this section, we show that the bootstrap works for the covariance operator estimates $\hat{\Gamma}_n, \hat{C}_n$, too. We compare them with their bootstrap analogues
$$
\hat{\Gamma}_n^* = \frac{1}{n} \sum_{t=0}^{n-1} X_t^* \otimes X_t^* , \quad \quad \hat{C}_n^* = \frac{1}{n} \sum_{t=0}^{n-1} X_t^* \otimes X_{t+1}^*.
$$
We again consider the Mallows metric, which, for bounded linear operators $A, B: \mathcal{H} \to \mathcal{H}$, we define with respect to the operator norm $\Vert . \Vert_\mathcal{L}$:
$$
d_2^2(A,B) = \inf_{A', B'} \ \E \Vert A'-B' \Vert^2_\mathcal{L} ,
$$
where the infimum is taken over all random operators $A'$ and $B'$ with the same marginal distribution as $A$ resp. $B$.

Note that $\hat{\Gamma}_n$ is an unbiased estimate of $\Gamma$ as $\E X_t = 0$. In the bootstrap world, we have an analogous property asymptotically. More precisely, we show in Lemma \ref{lemma5} that $\E^* \hat{\Gamma}_n^* =  \hat{\Gamma}_n + O_p(\frac{1}{n})$. Therefore, we have to compare the estimation error $\hat{\Gamma}_n -\Gamma$ with $\hat{\Gamma}_n^*-\hat{\Gamma}_n$.


\begin{thm}\label{bootcov}
Under the assumptions of Theorem \ref{bootmean}, we have for $n \to \infty$
\begin{equation*}
n d_2^2\left(\hat{\Gamma}_n-\Gamma, \hat{\Gamma}_n^*-\hat{\Gamma}_n\right)\underset{p}\rightarrow 0
\end{equation*}
\end{thm}

The theorem, in particular, implies that $\sqrt{n} (\hat{\Gamma}_n-\Gamma)$ and, conditional on $X_0, \ldots, X_n$, $\sqrt{n}(\hat{\Gamma}_n^*-\hat{\Gamma}_n)$ have the same asymptotic distribution by Lemma 8.3 of \citep{Bickel}.\\

For the lag-1 autocovariance operator, we have, again from Lemma \ref{lemma5}, that $\E^* \hat{C}_n^* =  \hat{C}_n \hat{\Pi}_{k_n} + O_p(\frac{1}{n})$ where $\hat{\Pi}_{k_n}$ denotes the projection onto the span of the first $k_n$ eigenvectors of $\hat{\Gamma}_n$. So, this provides the appropriate reference point in the bootstrap world if we want to approximate the distribution of the estimation error $\hat{C}_n - C$. More precisely, \\

\begin{thm}\label{bootacov}
Under the assumptions of Theorem \ref{bootmean}, we have for $n \to \infty$
\begin{equation*}
n d_2^2\left(\hat{C}_n-C, \hat{C}_n^*-\hat{C}_n \hat{\Pi}_{k_n} \right)\underset{p}\rightarrow 0
\end{equation*}
\end{thm}


\section{Appendix - Technical Lemmas and Proofs}

Throughout this section, 
$$
\Pi_p = \sum_{j=1}^p v_j\otimes v_j, \quad \quad \hat{\Pi}_{p} = \sum_{j=1}^p \hat{v}_j\otimes \hat{v}_j
$$
denote the projections onto the span of the first $p$ orthonormal eigenfunctions $v_1, \ldots, v_p$ resp. empirical eigenfunctions $\hat{v}_1, \ldots, \hat{v}_p$. As the eigenfunctions are only uniquely determined up to their sign, we have to compare later on $v_j$ with $\hat{c}_j \hat{v}_j$ where 
$$\hat{c}_j = \text{sgn}(\langle \hat{v}_j, v_j \rangle).$$
The first two auxiliary results have been essentially used already by \citep{Mas}. We defer their proofs to the supplement \ref{supp}.

\begin{lemma}\label{lemma1}
$\hat{\Pi}_{k_n}=\hat{\Gamma}_n\hat{\Gamma}_n^{\dagger}=\hat{\Gamma}_n^{\dagger}\hat{\Gamma}_n, \ \hat{\Psi}_n\hat{\Pi}_{k_n}=\hat{\Psi}_n$.
\end{lemma}

\begin{lemma} \label{lemma2}
$\hat{\Psi}_n-\Psi\hat{\Pi}_p=\dfrac{1}{n}S_n\hat{\Gamma}_n^{\dagger}$ with $S_n=\sum_{t=1}^n X_{t-1} \otimes \epsilon_t = n\left(\hat{C}_n-\Psi\hat{\Gamma}_n\right)$.
\end{lemma}

Next we state that the well-known strong consistency of $\hat{\Psi}_n$ as an estimate of $\Psi$ in particular holds under our set of assumptions, and we collect some immediate consequences for reference.

\begin{lemma}\label{lemma4}
Let $||\Psi||_{\mathcal{L}} < \hat{\delta} < 1$. Under the conditions of Theorem \ref{bootmean}, we have 

a) $\left\Vert\hat{\Psi}_n-\Psi\right\Vert_{\mathcal{L}}\underset{a.s.}\rightarrow 0 \quad$ for $n \to \infty$.

b) $\left\Vert\hat{\Psi}_n\right\Vert_{\mathcal{L}} \le \hat{\delta}\quad $ for all large enough $n$,

c) $\left\Vert \hat{\Psi}_n^k - \Psi^k\right\Vert_{\mathcal{L}} = \hat{\delta}^k \left\Vert\Psi-\hat{\Psi}_n\right\Vert_{\mathcal{L}} O_p(1)$.
\end{lemma}
\begin{proof}
a) The result is a slight modification of Theorem 8.7 of \citep{Bosq}, and the proof is defered to the supplement \ref{supp}.

b) From a) we immediately have $
\Vert\hat{\Psi}_n\Vert_{\mathcal{L}} \leq \Vert\Psi\Vert_{\mathcal{L}} +\Vert\hat{\Psi}_n-\Psi\Vert_{\mathcal{L}}  \leq \hat{\delta}$ for large enough $n$.

c) First, we note that
\begin{equation*}
\hat{\Psi}_n^k - \Psi^k = \left(\hat{\Psi}_n-\Psi\right)\sum_{j=0}^{k-1}\Psi^j\hat{\Psi}_n^{k-1-j}
\end{equation*}
The assertion follows from, using b) and  $||\Psi||_{\mathcal{L}} < \hat{\delta}$,
\begin{eqnarray*}
\left\Vert\sum_{j=0}^{k-1}\Psi^j\hat{\Psi}_n^{k-1-j}\right\Vert_{\mathcal{L}} &\leq & \sum_{j=0}^{k-1}\left\Vert\Psi^j\right\Vert_{\mathcal{L}} \left\Vert\hat{\Psi}_n^{k-1-j}\right\Vert _{\mathcal{L}} \leq \sum_{j=0}^{k-1}\left\Vert\Psi\right\Vert^j_{\mathcal{L}}\left\Vert\hat{\Psi}_n\right\Vert^{k-1-j}_{\mathcal{L}} \\
&\leq & \hat{\delta}^{k-1}\sum_{j=0}^{k-1}\left(\dfrac{\left\Vert\Psi\right\Vert_{\mathcal{L}}}{\hat{\delta}}\right)^j \leq \dfrac{\hat{\delta}^{k-1}}{1-\dfrac{\left\Vert\Psi\right\Vert_{\mathcal{L}}}{\hat{\delta}}}=\dfrac{\hat{\delta}^k}{\hat{\delta}-\left\Vert\Psi\right\Vert_{\mathcal{L}}}
\end{eqnarray*}
for all large enough $n$.
\end{proof}

\begin{proof} ({\bf Theorem \ref{theorem2}})\\
Let $F_n$ denote the empirical distribution of $\epsilon_1,\ldots,\epsilon_n$. Then, from Lemma 8.4 of \citep{Bickel}, we have $d_2\left(F_n,F\right) \rightarrow 0$ a.s. Hence it suffices to show that $d_2\left(F_n,\hat{F}_n\right)\underset{p}\rightarrow 0$. Let
$$
U_0=\epsilon_J, \ V_0=\tilde{\epsilon}_J=\hat{\epsilon}_J-\dfrac{1}{n}\sum_{j=1}^n\hat{\epsilon}_j, \\
$$
where $J$ is Laplace distributed on $\{1,\ldots,n\}$, i.e. $pr(J=t)=\dfrac{1}{n}, \ 1\leq t\leq n$. The random variables $U_0, V_0$ have marginal distributions $F_n$ respectively $\hat{F}_n$. As in the proof of Theorem 3.1 of \citep{Franke}, we have from the definition of the Mallows metric
$$
d_2^2\left(F_n,\hat{F}_n\right) \leq  \dfrac{1}{n}\sum_{k=1}^n\left\Vert \epsilon_k - \hat{\epsilon}_k + \dfrac{1}{n}\sum_{j=1}^n\hat{\epsilon}_j \right\Vert^2
\leq  \dfrac{6}{n}\sum_{k=1}^n\left\Vert \hat{\epsilon}_k-\epsilon_k \right\Vert^2 + \dfrac{3}{n^2}\left\Vert \sum_{j=1}^n\epsilon_j \right\Vert^2
$$
From the law of large numbers for i.i.d. random variables we have 
\begin{equation*}
\dfrac{1}{n}\sum_{j=1}^n\epsilon_j \underset{p}\rightarrow \E\epsilon_j=0, \ n\to \infty
\end{equation*}
such that the second term on the right-hand side vanishes for $n\to \infty$. For the first term, we show in the following parts a)-c) of the proof
\begin{equation*}
\left\Vert \hat{\epsilon}_t-\epsilon_t \right\Vert^2 \leq \left\Vert X_{t-1} \right\Vert^2R_n +3\left\Vert \Pi_{k_n}\left(X_{t-1}\right)-X_{t-1} \right\Vert^2
\end{equation*}
where $R_n$ does not depend on $t$, and $R_n\underset{p}\rightarrow 0$. Hence, for $n \to \infty$,
$$
\dfrac{1}{n}\sum_{t=1}^n\left\Vert \hat{\epsilon}_t-\epsilon_t \right\Vert^2 \leq \dfrac{1}{n}\sum_{t=1}^n\left\Vert X_{t-1} \right\Vert^2R_n + 3 \dfrac{1}{n}\sum_{t=1}^n\left\Vert \Pi_{k_n}\left(X_{t-1}\right)-X_{t-1} \right\Vert^2 \underset{p}\rightarrow 0,
$$
as, by Corollary 6.2 of \citep{Bosq}, $\dfrac{1}{n}\sum_{t=1}^n\left\Vert X_{t-1} \right\Vert^2 \underset{p}\rightarrow \E\left\Vert X_1 \right\Vert^2 < \infty$, and, by stationarity of $\{X_t\}$
\begin{equation*}
\E\left(\dfrac{1}{n}\sum_{t=1}^n\left\Vert \Pi_{k_n}\left(X_{t-1}\right)- X_{t-1} \right\Vert^2 \right)=\E\sum_{j={k_n}+1}^\infty \left\langle X_1,v_j \right\rangle^2 \to 0
\end{equation*}
for $k_n \to \infty$, using a monotone convergence argument and $\E\sum_{j=1}^\infty \left\langle X_1,v_j \right\rangle^2 = \E\left\Vert X_1 \right\Vert^2 < \infty$.

\textbf{a}) By definition of $\epsilon_t, \ \hat{\epsilon}_t$, we have
\begin{eqnarray*}
\left\Vert \epsilon_t-\hat{\epsilon}_t \right\Vert^2 &=& \left\Vert X_t-\Psi\left(X_{t-1}\right)-X_t+\hat{\Psi}_n\left(X_{t-1}\right)\right\Vert^2 \\
&=& \left\Vert \left(\hat{\Psi}_n-\Psi\right)\left(X_{t-1}\right) \right\Vert^2 \\
&=& \left\Vert \left(\hat{\Psi}_n-\Psi\hat{\Pi}_{k_n}\right)\left(X_{t-1}\right) + \Psi\left(\hat{\Pi}_{k_n}-\Pi_{k_n}\right)\left(X_{t-1}\right) + \Psi\left( \Pi_{k_n}\left(X_{t-1}\right)-X_{t-1}\right)\right\Vert^2 \\
&\leq & 3\left\Vert \left(\hat{\Psi}_n-\Psi\hat{\Pi}_{k_n}\right)\left(X_{t-1}\right)\right\Vert^2 + 3\left\Vert \left(\hat{\Pi}_{k_n}-\Pi_{k_n}\right)\left(X_{t-1}\right)\right\Vert^2 + 3\left\Vert \Pi_{k_n}\left(X_{t-1}\right)-X_{t-1}\right\Vert^2
\end{eqnarray*}
using $\left\Vert \Psi \right\Vert_{\mathcal{L}} <1$. We now show that the first and the second terms are bounded in the required manner.

\textbf{b}) We split $\left(\hat{\Pi}_{k_n}-\Pi_{k_n}\right)\left(X_{t-1}\right)$ into two terms
\begin{eqnarray*}
\left(\hat{\Pi}_{k_n}-\Pi_{k_n}\right)\left(X_{t-1}\right) &=& \sum_{j=1}^{k_n}\left\langle X_{t-1},\hat{v}_j\right\rangle \hat{v}_j - \sum_{j=1}^{k_n}\left\langle X_{t-1},v_j\right\rangle v_j \\
&=& \sum_{j=1}^{k_n}\left\langle X_{t-1},\hat{c}_j\hat{v}_j\right\rangle \left(\hat{c}_j\hat{v}_j-v_j\right) + \sum_{j=1}^{k_n}\left\langle X_{t-1},\hat{c}_j\hat{v}_j-v_j \right\rangle v_j.
\end{eqnarray*}
As $v_1,v_2,\ldots$ are orthonormal, we have for the second term
\begin{eqnarray*}
\left\Vert \sum_{j=1}^{k_n}\left\langle X_{t-1},\hat{c}_j\hat{v}_j-v_j \right\rangle v_j \right\Vert^2 &=& \sum_{j=1}^{k_n}\left\langle X_{t-1},\hat{c}_j\hat{v}_j-v_j \right\rangle^2 \\
&\leq & \left\Vert X_{t-1}\right\Vert^2\sum_{j=1}^{k_n}\left\Vert \hat{c}_j\hat{v}_j-v_j \right\Vert^2
\end{eqnarray*}
where the right hand side converges to 0 in probability, as, from the remarks after Theorem 16.1 of \citep{Horvath} and (\ref{ratekn})
\begin{equation*}
\E\sum_{j=1}^{k_n}\left\Vert \hat{c}_j\hat{v}_j-v_j \right\Vert^2 = \dfrac{1}{n}\sum_{j=1}^{k_n}\dfrac{1}{a_j^2}\, O(1) \to 0 \quad \quad \text{for } n\to \infty.
\end{equation*}

For the first term, we have, as $\left\Vert \hat{c}_j\hat{v}_j\right\Vert=1$,
\begin{eqnarray*}
\left\Vert \sum_{j=1}^{k_n}\left\langle X_{t-1},\hat{c}_j\hat{v}_j\right\rangle \left(\hat{c}_j\hat{v}_j-v_j\right)\right\Vert^2 &\leq &k_n\sum_{j=1}^{k_n}\left\langle X_{t-1},\hat{c}_j\hat{v}_j\right\rangle^2\left\Vert \hat{c}_j\hat{v}_j-v_j\right\Vert^2 \\
&\leq & \left\Vert X_{t-1}\right\Vert^2 k_n\sum_{j=1}^{k_n}\left\Vert \hat{c}_j\hat{v}_j-v_j\right\Vert^2
\end{eqnarray*}
where again the right hand side converges to 0 in probability as, from above,
\begin{equation*}
\E k_n\sum_{j=1}^n\left\Vert \hat{c}_j\hat{v}_j-v_j\right\Vert^2 = \dfrac{k_n}{n}\sum_{j=1}^{k_n}\dfrac{1}{a_j^2}\, O(1) \to 0 \quad \quad \text{for } n\to \infty.
\end{equation*}

\textbf{c}) Using Lemma \ref{lemma2}, we have
\begin{eqnarray*}
\left\Vert\left(\hat{\Psi}_n-\Psi\hat{\Pi}_{k_n}\right)\left(X_{t-1}\right)\right\Vert^2 &=& \left\Vert\dfrac{1}{n}S_n\hat{\Gamma}_n^{\dagger}\left(X_{t-1}\right)\right\Vert^2 \\
&\leq & \left\Vert\dfrac{1}{n}S_n\right\Vert_{\mathcal{L}}^2\left\Vert \hat{\Gamma}_n^{\dagger}\left(X_{t-1}\right)\right\Vert^2 \\
&=& \left\Vert\dfrac{1}{n}S_n\right\Vert_{\mathcal{L}}^2\left\Vert \sum_{j=1}^{k_n}\dfrac{1}{\hat{\lambda}_j} \langle X_{t-1}, \hat{v}_j\rangle \ \hat{v}_j\right\Vert^2 \\
&\leq & \left\Vert\dfrac{1}{n}S_n\right\Vert_{\mathcal{L}}^2 \left\Vert X_{t-1}\right\Vert^2\sum_{j=1}^{k_n}\dfrac{1}{\hat{\lambda}_j^2}
\end{eqnarray*}
using the Cauchy-Schwarz inequality. Moreover, as $C=\Psi\Gamma$ and $\left\Vert\Psi\right\Vert_{\mathcal{L}} \leq 1$, 
\begin{eqnarray*}
\left\Vert\dfrac{1}{n}S_n\right\Vert_{\mathcal{L}}^2 &=& \left\Vert\hat{C}_n-\Psi\hat{\Gamma}_n\right\Vert^2 \leq 2\left\Vert\hat{C}_n-C\right\Vert_{\mathcal{L}}^2 + 2\left\Vert\Psi\left(\Gamma-\hat{\Gamma}_n\right)\right\Vert_{\mathcal{L}}^2 \\
&\leq & 2\left\Vert\hat{C}_n-C\right\Vert_{\mathcal{L}}^2 +2\left\Vert\hat{\Gamma}_n-\Gamma\right\Vert_{\mathcal{L}}^2 = O_p\left(\dfrac{1}{n}\right),
\end{eqnarray*}
as, from the remarks after Theorem 16.1 of \citep{Horvath}, we have $\E\left\Vert\hat{\Gamma}_n-\Gamma\right\Vert_{\mathcal{L}}^2=O\left(\dfrac{1}{n}\right)$, and from Theorem 3 of \citep{Pumo}, analogously $\E\left\Vert\hat{C}_n-C\right\Vert_{\mathcal{L}}^2=O\left(\dfrac{1}{n}\right)$.\\

As $\left\Vert \cdot \right\Vert_{\mathcal{L}} \leq \left\Vert \cdot \right\Vert_{\mathcal{S}}$, we get from (\ref{ratekn}) and Theorem 4.1 of \citep{Bosq} with $D$ denoting some suitable constant
\begin{equation*}
\dfrac{1}{\lambda_{k_n}}\left\Vert \hat{\Gamma}_n-\Gamma \right\Vert_{\mathcal{L}} \leq D\ \frac{n^{1/4}}{(\log n)^{\beta}}\ \left\Vert \hat{\Gamma}_n-\Gamma \right\Vert_{\mathcal{L}} \underset{a.s.}\rightarrow 0 \ \text{for} \ n \to \infty.
\end{equation*}
Therefore, we have for all large enough $n$, $\left\Vert \hat{\Gamma}_n-\Gamma \right\Vert_{\mathcal{L}} \leq \dfrac{1}{2}\lambda_{k_n} \ a.s.$
and, as in the proof of Theorem 8.7 of \citep{Bosq},
\begin{equation} \label{lowlambdahat}
\hat{\lambda}_{k_n} \geq \lambda_{k_n}-\left\Vert \hat{\Gamma}_n-\Gamma \right\Vert_{\mathcal{L}} \geq \dfrac{1}{2}\lambda_{k_n} \ a.s.
\end{equation}
using $\sup_{j\geq 1}\left\vert\hat{\lambda}_j-\lambda_j\right\vert \leq \left\Vert \hat{\Gamma}_n-\Gamma \right\Vert_{\mathcal{L}}$. Therefore,  for large enough $n$, using (\ref{ratekn}) again,
$$
\left\Vert\dfrac{1}{n}S_n\right\Vert_{\mathcal{L}}^2\sum_{j=1}^{k_n}\dfrac{1}{\hat{\lambda}_j^2} \leq  4\left\Vert\dfrac{1}{n}S_n\right\Vert_{\mathcal{L}}^2\sum_{j=1}^{k_n}\dfrac{1}{\lambda_j^2} \le 4\left\Vert\dfrac{1}{n}S_n\right\Vert_{\mathcal{L}}^2\sum_{j=1}^{k_n}\dfrac{1}{a_j^2} = o_p\left(\dfrac{1}{n}\right).
$$
\end{proof}

\begin{proof} ({\bf Theorem \ref{bootmean}})\\
As in the proof of Theorem 4.1 of \citep{Franke}, we choose a particular realization of innovation pairs $\left(\epsilon_t',\epsilon_t^*\right)$ such that \\
i) $\left(\epsilon_t',\epsilon_t^*\right)$ i.i.d. conditional on $X_0, \ldots, X_n$,\\
ii) the marginal distributions of $\epsilon_t'$ and $\epsilon_t^*$ are $F$ resp. $\hat{F}_n$,\\
iii) $\E^*\left\Vert\epsilon_t'-\epsilon_t^*\right\Vert^2=d_2^2\left(F,\hat{F}_n\right)$. \\
The latter can be achieved by Lemma 8.1. of \citep{Bickel}. 
Moreover, we choose $X_0'$ distributed as, but independent of $X_0$ and of $\left(\epsilon_t',\epsilon_t^*\right), \ t\geq 1$. Finally, we choose $X_0^*=X_0'$, and we set 
\begin{equation*}
X_t'=\Psi\left(X_{t-1}'\right)+\epsilon_t', \quad X_t^*=\hat{\Psi}_n\left(X_{t-1}^*\right)+\epsilon_t^*, \; t \geq 1.
\end{equation*}
$X_0', \ldots, X_n'$ is a independent realization of the data $X_0, \ldots, X_n$, and $X_0^*, \ldots, X_n^*$ is a realization of the bootstrap data of section \ref{secboot}. If we iterate the recursions, we get a representation of $X_t', X_t^*$ in terms of $X_0'$ and the innovations:
\begin{equation} \label{finMA}
X_t'=\Psi^t\left(X_0'\right)+\sum_{k=1}^{t} \Psi^{t-k}(\epsilon_k'), \quad X_t^*=\hat{\Psi}_n^t\left(X_0'\right)+\sum_{k=1}^{t} \hat{\Psi}_n^{t-k}(\epsilon_k^*), \; t \geq 1
\end{equation}
\textbf{a)} As $\E X'_0 = \E X_0 = 0$ and, by definition, $\E^* \epsilon_t^* =0$, we get, using linearity of the autoregressive operator,
$$
\E^* X_t^* = \E^* \bigg(\hat{\Psi}_n^t(X'_0) + \sum_{k=1}^t \hat{\Psi}_n^{t-k}(\epsilon_k^*) \bigg) = 0
$$
immediately from (\ref{finMA}).\\

\textbf{b}) We have to consider
$$
n d_2^2\left(\bar{X}_n, \bar{X}_n^* \right) \le n \E^* || \bar{X}'_n - \bar{X}_n^*||^2 = \frac{1}{n} \sum_{t,s=0}^{n-1} \E^* \langle  X'_t - X_t^*, X'_s - X_s^* \rangle ,
$$
where $\bar{X}'_n$ denotes the sample mean of $X_0', \ldots, X_{n-1}'$. According to (\ref{finMA}), we split the differences into 3 parts $X'_t - X_t^* = a_t + b_t + c_t$, i.e.
\begin{equation*}
X_t'-X_t^*= + \sum_{k=1}^t \Psi^{t-k}\left(\epsilon_{k}'-\epsilon_{k}^*\right) + \sum_{k=1}^t \left(\Psi^{t-k}-\hat{\Psi}_n^{t-k}\right)\left(\epsilon_k^*\right) = a_t + b_t + c_t.
\end{equation*}
So, we have to study
$$
\frac{1}{n} \sum_{t,s=0}^{n-1} \E^* \langle  a_t + b_t + c_t, a_s + b_s + c_s \rangle . 
$$
We show in the following three parts of the proof that the terms 
\begin{equation} \label{abc}
\frac{1}{n} \sum_{t,s=0}^{n-1} \E^* \langle a_t, c_s\rangle , \; \frac{1}{n} \sum_{t,s=0}^{n-1} \E^* \langle b_t, b_s\rangle  \; \text{ and }\; \frac{1}{n} \sum_{t,s=0}^{n-1} \E^* \langle c_t, c_s\rangle  
\end{equation}
are of order $o_p(1)$. The remaining terms can be handled analogously, and the assertion follows.\\

\textbf{c}) Due to independence of $X_0'$ and $\epsilon_k^*$ for $k \ge 1$, and the fact that their mean is 0,
\begin{eqnarray*}
\E^* \langle a_t, c_s\rangle  &=& \sum_{k=1}^s \E^* \langle  \left(\Psi^t-\hat{\Psi}_n^t\right)\left(X_0'\right), \left(\Psi^{s-k}-\hat{\Psi}_n^{s-k}\right)\left(\epsilon_k^*\right)\rangle  \\
&=& \sum_{k=1}^s \langle  \E^* \left(\Psi^t-\hat{\Psi}_n^t\right)\left(X_0'\right) , \E^* \left(\Psi^{s-k}-\hat{\Psi}_n^{s-k}\right)\left(\epsilon_k^*\right) \rangle \, =\, 0
\end{eqnarray*}
Therefore, the first term of (\ref{abc}) vanishes.\\

\textbf{d}) As $(\epsilon_{k}',\epsilon_{k}^*), k=1, \ldots, n,$ are independent with mean 0, we have for $s \le t$ and $||\Psi||_{\mathcal{L}} \le \hat{\delta} < 1$ 
\begin{eqnarray*}
\E^* \langle b_t, b_s\rangle  &=& \sum_{k=1}^t \sum_{l=1}^s \langle  \Psi^{t-k}\left(\epsilon_{k}'-\epsilon_{k}^*\right), \Psi^{s-l}\left(\epsilon_{l}'-\epsilon_{l}^*\right) \rangle \\
&=& \sum_{k=1}^s \E^* \langle  \Psi^{t-k}\left(\epsilon_{k}'-\epsilon_{k}^*\right), \Psi^{s-k}\left(\epsilon_{k}'-\epsilon_{k}^*\right) \rangle\\
&\le& \sum_{k=1}^s ||\Psi^{t-k}||_{\mathcal{L}} ||\Psi^{s-k}||_{\mathcal{L}} \E^* ||\epsilon_{k}'-\epsilon_{k}^*||^2\\
&\le& \sum_{k=1}^s \hat{\delta}^{t+s-2k} d_2^2(F, \hat{F}_n)\\
&=& \hat{\delta}^{t-s} \sum_{k=1}^s \hat{\delta}^{2(s-k)} d_2^2(F, \hat{F}_n) \le \frac{\hat{\delta}^{t-s}}{1 - \hat{\delta}^2 } d_2^2(F, \hat{F}_n) 
\end{eqnarray*}
We conclude
\begin{eqnarray*}
\frac{1}{n} \sum_{t,s=0}^{n-1} \E^* \langle b_t, b_s\rangle &\le& \frac{2}{n} \sum_{t=0}^{n-1} \sum_{s=0}^t \hat{\delta}^{t-s} d_2^2(F, \hat{F}_n)\ \frac{1}{1 - \hat{\delta}^2 } \\
&\le& \frac{2}{n} \sum_{t=0}^{n-1} \frac{1}{1-\hat{\delta}} d_2^2(F, \hat{F}_n) \ \frac{1}{1 - \hat{\delta}^2 } = o_p(1)
\end{eqnarray*}
by Theorem \ref{theorem2}.\\

\textbf{e}) From Theorem \ref{theorem2} and Lemma 8.3 of \citep{Bickel}
\begin{equation} \label{Eeps2}
\E^*\left\Vert\epsilon_1^*\right\Vert^2=\dfrac{1}{n}\sum_{t=1}^n\left\Vert\hat{\epsilon}_t\right\Vert^2 \underset{p}\rightarrow \E^*\left\Vert\epsilon_1\right\Vert^2,
\end{equation}
i.e. $\E^*\left\Vert\epsilon_t^*\right\Vert^2 = O_p(1)$. As $\epsilon_{k}^*, k=1, \ldots, n,$ are independent, we have for $s \le t$
\begin{eqnarray*}
\E^* \langle c_t, c_s\rangle &=& \sum_{k=1}^t \sum_{l=1}^s \E^* \langle  \left(\Psi^{t-k}-\hat{\Psi}_n^{t-k}\right)\left(\epsilon_k^*\right), \left(\Psi^{s-l}-\hat{\Psi}_n^{s-l}\right)\left(\epsilon_l^*\right) \rangle \\
&=& \sum_{k=1}^s \E^* \langle  \left(\Psi^{t-k}-\hat{\Psi}_n^{t-k}\right)\left(\epsilon_k^*\right), \left(\Psi^{s-k}-\hat{\Psi}_n^{s-k}\right)\left(\epsilon_k^*\right) \rangle \\
&\le& \sum_{k=1}^s \hat{\delta}^{t+s-2k} ||\Psi - \hat{\Psi}_n||_\mathcal{L}^2\ \E^* || \epsilon_k^*||^2\, O_p(1) \\
&=& \hat{\delta}^{t-s} \sum_{k=1}^s \hat{\delta}^{2(s-k)} ||\Psi - \hat{\Psi}_n||_\mathcal{L}^2\ O_p(1) = \hat{\delta}^{t-s} ||\Psi - \hat{\Psi}_n||_\mathcal{L}^2\ O_p(1)
\end{eqnarray*}
using (\ref{Eeps2}) and Lemma \ref{lemma4}, b). We conclude 
$$
\frac{1}{n} \sum_{t,s=0}^{n-1} \E^* \langle c_t, c_s\rangle \le \frac{2}{n} \sum_{t=0}^{n-1} \sum_{s=0}^t \hat{\delta}^{t-s} ||\Psi - \hat{\Psi}_n||_\mathcal{L}^2\, O_p(1) =||\Psi - \hat{\Psi}_n||_\mathcal{L}^2\, O_p(1) = o_p(1)
$$
by Lemma \ref{lemma4}, a).
\end{proof}

\begin{lemma} \label{lemma5}
Under the conditions of Theorem \ref{bootmean}, we have \\
a) $\E^*\hat{\Gamma}_n^*=\hat{\Gamma}_n + O_p\left(\dfrac{1}{n}\right)$\\
b) $\E^*\hat{C}_n^*=\hat{C}_n \hat{\Pi}_{k_n} + O_p\left(\dfrac{1}{n}\right)$
\end{lemma}
\begin{proof}
a) Plugging in the recursive representation (\ref{finMA}) of $X_t^*$ into the definition of $\hat{\Gamma}_n^*$, we get 
\begin{equation*} 
\begin{aligned}
\hat{\Gamma}_n^* &=\dfrac{1}{n}\sum_{t=0}^{n-1}\left( \hat{\Psi}_n^t\left(X_0'\right)\otimes \hat{\Psi}_n^t\left(X_0'\right) + \sum_{k=1}^t\hat{\Psi}_n^t\left(X_0'\right)\otimes \hat{\Psi}_n^{t-k}\left(\epsilon_k^*\right) \right.\\
&\quad\qquad \left. + \sum_{k=1}^t \hat{\Psi}_n^{t-k}\left(\epsilon_k^*\right)\otimes \hat{\Psi}_n^t\left(X_0'\right)+ \sum_{k,l=1}^t\hat{\Psi}_n^{t-k}\left(\epsilon_k^*\right)\otimes \hat{\Psi}_n^{t-l}\left(\epsilon_l^*\right) \right)
\end{aligned}
\end{equation*}
As $\E^*\epsilon_k^*=0$ and, hence, $\E^*\hat{\Psi}^l\left(\epsilon_k^*\right)=0$ due to linearity and as $\epsilon_1^*,\ldots,\epsilon_n^*,X_0'$ are independent, we get
\begin{equation} \label{boldstar}
\E^*\hat{\Gamma}_n^*=\dfrac{1}{n}\sum_{t=0}^{n-1}\left(\E^*\hat{\Psi}_n^t\left(X_0'\right)\otimes\hat{\Psi}_n^t\left(X_0'\right)+ \sum_{k=1}^t\E^*\hat{\Psi}_n^{t-k}\left(\epsilon_k^*\right)\otimes\hat{\Psi}_n^{t-k}\left(\epsilon_k^*\right) \right)
\end{equation}
As in the bootstrap world, $\hat{\Psi}_n^l$ are fixed operators, in view of (\ref{kronecker}) we have to investigate mainly $\E^*\epsilon_k^*\otimes\epsilon_k^*$.
\begin{equation*}
\E^*\epsilon_k^*\otimes\epsilon_k^*=\dfrac{1}{n}\sum_{t=1}^n\tilde{\epsilon}_t\otimes \tilde{\epsilon}_t=\dfrac{1}{n}\sum_{t=1}^n\left(\hat{\epsilon}_t-\bar{\hat{\epsilon}}_n\right) \otimes \left(\hat{\epsilon}_t-\bar{\hat{\epsilon}}_n\right)=\dfrac{1}{n}\sum_{t=1}^n\hat{\epsilon}_t \otimes \hat{\epsilon}_t - \bar{\hat{\epsilon}}_n \otimes \bar{\hat{\epsilon}}_n
\end{equation*}
with $\bar{\hat{\epsilon}}_n=\dfrac{1}{n}\sum_{k=1}^n\hat{\epsilon}_k$. As $\hat{\epsilon}_k=X_k-\hat{\Psi}_n\left(X_{k-1}\right)$,
\begin{eqnarray*}
\dfrac{1}{n}\sum_{t=1}^n\hat{\epsilon}_t \otimes \hat{\epsilon}_t &=&\dfrac{1}{n}\sum_{t=1}^n\left(X_t-\hat{\Psi}_n\left(X_{t-1}\right)\right)\otimes \left(X_t-\hat{\Psi}_n\left(X_{t-1}\right)\right) \\
&=& \hat{\Gamma}_n + \dfrac{1}{n}\left(X_n \otimes X_n - X_0 \otimes X_0\right)-\dfrac{1}{n}\sum_{t=1}^n\hat{\Psi}_n\left(X_{t-1}\right)\otimes X_t \\
& &- \dfrac{1}{n}\sum_{t=1}^n X_t \otimes \hat{\Psi}_n\left(X_{t-1}\right) + \hat{\Psi}_n\hat{\Gamma}_n\hat{\Psi}_n^T
\end{eqnarray*}
Using (\ref{kronecker}), the second and third terms are $-\hat{C}_n\hat{\Psi}_n^T$ and $-\hat{\Psi}_n\hat{C}_n^T$ respectively, such that, as $\hat{C}_n=\hat{\Psi}_n\hat{\Gamma}_n$
\begin{eqnarray*}
\dfrac{1}{n}\sum_{t=1}^n\hat{\epsilon}_t\otimes \hat{\epsilon}_t &=& \hat{\Gamma}_n-\hat{\Psi}_n\hat{\Gamma}_n\hat{\Psi}_n^T+\dfrac{1}{n}\left(X_n \otimes X_n - X_0 \otimes X_0\right) \\
&=& \hat{\Gamma}_n-\hat{\Psi}_n\hat{\Gamma}_n\hat{\Psi}_n^T + O_p\left(\dfrac{1}{n}\right)
\end{eqnarray*}
Similarly, we have
\begin{eqnarray*}
\bar{\hat{\epsilon}}_n \otimes \bar{\hat{\epsilon}}_n &=& \dfrac{1}{n^2}\sum_{k,l=1}^n \hat{\epsilon}_k \otimes \hat{\epsilon}_l=\dfrac{1}{n^2}\sum_{k,l=1}^n\left(X_k-\hat{\Psi}_n\left(X_{k-1}\right)\right)\otimes \left(X_l-\hat{\Psi}_n\left(X_{l-1}\right)\right) \\
&=& \bar{X}_{1:n}\otimes \bar{X}_{1:n} - \hat{\Psi}_n\left(\bar{X}_{1:n}\otimes \bar{X}_{0:(n-1)}\right)-\left( \bar{X}_{0:(n-1)}\otimes \bar{X}_{1:n}\right)\hat{\Psi}_n^T \\
&& \hat{\Psi}_n\left(\bar{X}_{0:(n-1)} \otimes \bar{X}_{0:(n-1)}\right)\bar{\Psi}_n^T
\end{eqnarray*}
where $\bar{X}_{1:n}, \ \bar{X}_{0:(n-1)}$ denote the sample means of $X_1,\ldots,X_n$ respectively $X_0,\ldots,X_{n-1}$. As, from Lemma \ref{lemma4}, $\left\Vert\hat{\Psi}_n-\Psi\right\Vert_{\mathcal{L}} \underset{a.s.}\rightarrow 0$ we have $\left\Vert\hat{\Psi}_n\right\Vert_{\mathcal{L}}=O_p(1)$, and as $\bar{X}_{0:(n-1)},\bar{X}_{1:n}$ are $O_p\left(\dfrac{1}{\sqrt{n}}\right)$ from the law of large numbers of FAR(1)-processes (compare Theorem 3.7 of \citep{Bosq}), we immediately get that $\bar{\hat{\epsilon}}_n \otimes \bar{\hat{\epsilon}}_n=O_p\left(\dfrac{1}{n}\right)$. So we get 
\begin{equation*}
\E^*\epsilon_k^* \otimes \epsilon_k^* = \hat{\Gamma}_n-\hat{\Psi}_n\hat{\Gamma}_n\hat{\Psi}_n^T + \dfrac{1}{n}R_n
\end{equation*}
with $R_n=O_p(1)$.
Hence, we have for the dominant term in $\E^*\hat{\Gamma}_n^*$
\begin{eqnarray*}
\E^*\sum_{k=1}^t\hat{\Psi}_n^{t-k}\epsilon_k^* \otimes \epsilon_k^*\left(\hat{\Psi}_n^{t-k}\right)^T &=& \sum_{k=1}^t\hat{\Psi}_n^{t-k}\left(\hat{\Gamma}_n-\hat{\Psi}_n\hat{\Gamma}_n\hat{\Psi}_n^T\right)\left(\hat{\Psi}_n^{t-k}\right)^T + \sum_ {k=1}^t\hat{\Psi}_n^{t-k}\dfrac{1}{n}R_n\left(\hat{\Psi}_n^{t-k}\right)^T \\
&=& \sum_{k=1}^t\hat{\Psi}_n^{t-k}\hat{\Gamma}_n\left(\hat{\Psi}_n^{t-k}\right)^T - \sum_{l=0}^{t-1}\hat{\Psi}_n^{t-l}\hat{\Gamma}_n\left(\hat{\Psi}_n^{t-l}\right)^T + O_p\left(\dfrac{1}{n}\right) \\
&=& \hat{\Gamma}_n-\hat{\Psi}_n^t\hat{\Gamma}_n\left(\hat{\Psi}_n^t\right)^T + O_p\left(\dfrac{1}{n}\right)
\end{eqnarray*}
where we have used that $R_n=O_p(1)$,  $\left\Vert\hat{\Psi}_n^l\right\Vert_{\mathcal{L}} \leq \left\Vert\hat{\Psi}_n\right\Vert_{\mathcal{L}}^l \leq \hat{\delta}^l$ for some $\hat{\delta}<1$ and large enough $n$ from Lemma \ref{lemma4}, and $\sum_{k=1}^t\hat{\delta}^{2(t-k)} \leq \dfrac{1}{1-\hat{\delta}^2}$. Finally,
\begin{eqnarray*}
\E^*\dfrac{1}{n}\sum_{t=0}^{n-1}\sum_{k=1}^t\hat{\Psi}_n^{t-k}\epsilon_k^* \otimes \epsilon_k^*\left(\hat{\Psi}_n^{t-k}\right)^T &=& \dfrac{1}{n}\sum_{t=0}^{n-1}\left(\hat{\Gamma}_n-\hat{\Psi}_n^t\hat{\Gamma}_n\left(\hat{\Psi}_n^t\right)^T\right) + O_p\left(\dfrac{1}{n}\right) \\
&=& \hat{\Gamma}_n + O_p\left(\dfrac{1}{n}\right)
\end{eqnarray*}
as, using again the above argument that $\left\Vert\hat{\Psi}_n^l\right\Vert_{\mathcal{L}} \leq \hat{\delta}_l$
\begin{equation*}
\left\Vert \sum_{t=0}^{n-1}\hat{\Psi}_n^t\hat{\Gamma}_n\left(\hat{\Psi}_n^t\right)^T\right\Vert_{\mathcal{L}} \leq \sum_{t=0}^{n-1}\hat{\delta}^{2t}\left\Vert\hat{\Gamma}_n \right\Vert_{\mathcal{L}} \leq \dfrac{1}{1-\hat{\delta}^2}\left\Vert\hat{\Gamma}_n \right\Vert_{\mathcal{L}}=O_p(1)
\end{equation*}
Using $\E^* X_0'\otimes X_0'=\E X_0\otimes X_0=\Gamma$, we get by the same kind of arguments that the first term in (\ref{boldstar}) is $O_p\left(\dfrac{1}{n}\right)$. \\

b) Using (\ref{finMA}), we have
$$
X_t^* \otimes X_{t+1}^* = \left( \hat{\Psi}_n^t(X_0') + \sum_{k=1}^t \hat{\Psi}_n^{t-k}(\epsilon_k^*) \right) \otimes \left( \hat{\Psi}_n^{t+1}(X_0') + \sum_{l=1}^{t+1} \hat{\Psi}_n^{t+1-l}(\epsilon_l^*) \right).
$$
Analogously to (\ref{boldstar}), we then conclude
$$
\E^*\hat{C}_n^*=\dfrac{1}{n}\sum_{t=0}^{n-1}\left(\E^*\hat{\Psi}_n^t\left(X_0'\right)\otimes\hat{\Psi}_n^{t+1}\left(X_0'\right)+ \sum_{k=1}^t\E^*\hat{\Psi}_n^{t-k}\left(\epsilon_k^*\right)\otimes\hat{\Psi}_n^{t+1-k}\left(\epsilon_k^*\right) \right). 
$$
From the same kind of calculations as in a), we get
$$
\E^* \hat{C}_n^* = \hat{\Psi}_n \hat{\Gamma}_n - \hat{\Psi}_n^{t+1} \hat{\Gamma}_n (\hat{\Psi}_n^t)^T + O_p(\frac{1}{n}) = \hat{\Psi}_n \hat{\Gamma}_n + O_p(\frac{1}{n}).
$$
As $\hat{\Psi}_n = \hat{C}_n \hat{\Gamma}_n^\dagger$ and, by Lemma \ref{lemma1}, $\hat{\Gamma}_n^\dagger \hat{\Gamma}_n = \hat{\Pi}_{k_n}$, we get the desired result.
\end{proof}

The next two lemmas just state a rule of calculation and an operator norm inequality needed in the following proof.

\begin{lemma}\label{lemma6}
If $\left(U,U^*\right), \ \left(V,V^*\right)$ are i.i.d. $L^2$-valued random variables such that $d_2^2\left(U,U^*\right)=\E \left\Vert U-U^*\right\Vert^2$; then 
\begin{equation*}
\E\left\Vert U\otimes V -U^*\otimes V^* \right\Vert^2_\mathcal{L} \leq 2\left(\E\left\Vert U\right\Vert^2 + \E\left\Vert U^*\right\Vert^2 \right)d_2^2\left(U,U^*\right) 
\end{equation*}
for any $x$ in $L^2$.
\end{lemma}
\begin{proof} From the definition of $\otimes$ and the Cauchy- Schwarz inequality, we have $\Vert x \otimes y \Vert_\mathcal{L} \le  \Vert x\Vert \Vert y\Vert$ such that
\begin{eqnarray*}
\left\Vert U\otimes V -U^*\otimes V^* \right\Vert^2_\mathcal{L} &=& \left\Vert \left(U-U^*\right) \otimes V  + U^*\otimes \left(V-V^*\right) \right\Vert^2_\mathcal{L} \\
&\leq& 2 \Vert U-U^*\Vert^2 \Vert V \Vert^2 + 2 \Vert U^*\Vert^2 \Vert V-V^* \Vert^2
\end{eqnarray*}
Using independence of $\left(U,U^*\right)$ and $\left(V,V^*\right)$
\begin{eqnarray*}
\E\left\Vert U\otimes V -U^*\otimes V^* \right\Vert^2 &\le& 
2 \E \Vert U-U^*\Vert^2 \E \Vert V \Vert^2 + 2 \E \Vert U^*\Vert^2 \E \Vert V-V^* \Vert^2\\
&\le& 2  d_2^2(U,U^*) \left(\E \left\Vert U\right\Vert^2 + \E \left\Vert U^*\right\Vert^2 \right)
\end{eqnarray*}
as $\E\left\Vert V-V^* \right\Vert^2=\E\left\Vert U-U^* \right\Vert^2$ and $\E\left\Vert V \right\Vert^2=\E\left\Vert U \right\Vert^2$.
\end{proof}

\begin{lemma}\label{lemmaHSnorm}
a) Let $A,B,S: \mathcal{H} \to \mathcal{H}$ be bounded linear operators where, in particular, $S$ is a Hilbert-Schmidt operator. Then, $ASB$ is a Hilbert-Schmidt operator and
\begin{equation} \label{opnorm}
\Vert ASB \Vert_\mathcal{S} \le \Vert A \Vert_\mathcal{L} \Vert S \Vert_\mathcal{S} \Vert B \Vert_\mathcal{L} .
\end{equation} 
b) For $x, y \in \mathcal{H}$, $x \otimes y$ is a Hilbert-Schmidt operator with 
\begin{equation} \label{kronnorm}
\Vert x \otimes y \Vert_\mathcal{S} = \Vert x \Vert\ \Vert y \Vert .
\end{equation}
\end{lemma}
\begin{proof} a) $ASB$ is a Hilbert-Schmidt operator by Lemma 16.7 of \citep{Meise}. From their Lemma 16.6, we get that the singular values of $ASB$ can be bounded by the product of the operator norms of $A$ and $B$ and the singular values of $S$. This implies (\ref{opnorm}) as the squared Hilbert-Schmidt norm is the sum of the squared singular values.

b) follows immediately from the definition of the Hilbert-Schmidt norm.
\end{proof} 

\begin{proof} (\textbf{Theorem \ref{bootcov}})\\
We choose $X_0^*=X_0'$ and $\left(\epsilon_t',\epsilon_t^*\right), t=1, \ldots, n,$ as in the proof of Theorem \ref{bootmean}. Let $\hat{\Gamma}_n'$ denote the sample covariance operator calculated from $X_0', \ldots, X_{n-1}'$. Due to stationarity of $\{X_t'\}$, we have $\E^*\hat{\Gamma}_n'=\E\hat{\Gamma}_n = \Gamma$, and, from Lemma \ref{lemma5}, $\E^*\hat{\Gamma}_n^*=\hat{\Gamma}_n+O_p\left(\dfrac{1}{n}\right)$. Hence, up to terms of order $\dfrac{1}{n}$,
\begin{eqnarray*}
& &\left(\hat{\Gamma}_n'-\Gamma\right) -\left(\hat{\Gamma}_n^*-\hat{\Gamma}_n\right) = \dfrac{1}{n} \sum_{t=0}^{n-1}A_t + O_p\left(\dfrac{1}{n}\right) \quad \quad \text{with}\\
& &A_t = X_t'\otimes X_t'-\E^*\left(X_t'\otimes X_t'\right)-\left(X_t^*\otimes X_t^*-\E^*\left(X_t^*\otimes X_t^*\right)\right),
\end{eqnarray*}
Using the recursive representation (\ref{finMA}) of $X_t',X_t^*$ and (\ref{kronecker}), we decompose $A_t=a_t + b_t + b_t^T + c_t +d_t$ with
\begin{eqnarray*}
a_t &=& \Psi^t\left[X_0' \otimes X_0'-\E^* \left(X_0' \otimes X_0'\right)\right]\left(\Psi^t\right)^T - \hat{\Psi}_n^t\left[X_0' \otimes X_0'-\E^* \left(X_0' \otimes X_0'\right)\right]\left(\hat{\Psi}_n^t\right)^T \\
b_t &=& \sum_{k=1}^t\left[\Psi^{t-k}\left(X_0' \otimes \epsilon_k' \right)\left(\Psi^t\right)^T - \hat{\Psi}_n^{t-k}\left(X_0' \otimes \epsilon_k^* \right)\left(\hat{\Psi}_n^t\right)^T\right] \\
c_t &=& \sum_{k\neq l=1}^t\left[\Psi^{t-l}\left(\epsilon_k'\otimes\epsilon_l'\right)\left(\Psi^{t-k}\right)^T- \hat{\Psi}_n^{t-l}\left(\epsilon_k^*\otimes\epsilon_l^*\right)\left(\hat{\Psi}_n^{t-k}\right)^T\right] \\
d_t &=& \sum_{k=1}^t\left[\Psi^{t-k}\left(\epsilon_k'\otimes\epsilon_k'\right)-\E^* \left(\epsilon_k'\otimes\epsilon_k'\right)\left(\Psi^{t-k}\right)^T - \hat{\Psi}_n^{t-k}\left(\epsilon_k^*\otimes\epsilon_k^*\right)+\E^*\left(\epsilon_k^*\otimes\epsilon_k^*\right)\left(\hat{\Psi}_n^{t-k}\right)^T\right]
\end{eqnarray*}
where we have used that $\left(\epsilon_k',\epsilon_k^*\right)$ are i.i.d. with mean 0 to get, e.g., $\E^*\epsilon_k'\otimes\epsilon_l'=0$ for $k\neq l$.

By definition of the Mallows metric and from $\Vert.\Vert_\mathcal{L} \le \Vert.\Vert_\mathcal{S}$, 
\begin{eqnarray}
& &d_2^2\left(\hat{\Gamma}_n-\Gamma, \hat{\Gamma}_n^*-\hat{\Gamma}_n\right) \le 
\E^*\left\Vert\left(\hat{\Gamma}_n'-\Gamma \right) - \left(\hat{\Gamma}_n^*-\hat{\Gamma}_n\right)\right\Vert^2_\mathcal{L} \nonumber \\
&\le& 2 \E^* \left\Vert \left(\hat{\Gamma}_n'-\E^* \hat{\Gamma}'_n \right) -\left(\hat{\Gamma}_n^*- \E^* \hat{\Gamma}_n^*\right) \right\Vert^2_\mathcal{L} + 2 \left\Vert \E^*  \hat{\Gamma}_n^* - \hat{\Gamma}_n \right\Vert^2_\mathcal{L} \nonumber \\
&\le& 2 \E^* \left\Vert \left(\hat{\Gamma}_n'-\E^* \hat{\Gamma}'_n \right) -\left(\hat{\Gamma}_n^*- \E^* \hat{\Gamma}_n^*\right) \right\Vert^2_\mathcal{S} + 2 \left\Vert \E^*  \hat{\Gamma}_n^* - \hat{\Gamma}_n \right\Vert^2_\mathcal{L} \nonumber \\
&=& \dfrac{2}{n^2}\sum_{s,t=0}^{n-1} \E^* \left\langle A_t,A_s\right\rangle_\mathcal{S} +  O_p\left(\dfrac{1}{n^2}\right), \label{d2Gamma}
\end{eqnarray}
using Lemma \ref{lemma5}. Hence, we have to study terms like 
$$
\E^*\sum_{s,t=0}^{n-1}\left\langle a_t,b_s\right\rangle_\mathcal{S}, \ \E^*\sum_{s,t=0}^{n-1}\left\langle c_t,c_s\right\rangle_\mathcal{S} \text{ or } \E^*\sum_{s,t=0}^{n-1}\left\langle d_t,d_s \right\rangle_\mathcal{S} . $$

\textbf{a}) We start with $\sum_{s,t=0}^{n-1}\E^*\left\langle c_t,c_s\right\rangle_\mathcal{S} = \sum_{s,t=0}^{n-1}\sum_{k\neq l=1}^t\sum_{i\neq j=1}^s\E^* B_{klij}^{(s,t)}$ where
\begin{equation*}
\begin{aligned}
B_{klij}^{(s,t)} &=\left\langle\Psi^{t-l}\left(\epsilon_k'\otimes\epsilon_l'\right)\left(\Psi^{t-k}\right)^T - \hat{\Psi}_n^{t-l}\left(\epsilon_k^*\otimes\epsilon_l^*\right)\left(\hat{\Psi}_n^{t-k}\right)^T,\right. \\
&\hspace{10mm} \left. \Psi^{s-j}\left(\epsilon_i'\otimes\epsilon_j'\right)\left(\Psi^{s-i}\right)^T - \hat{\Psi}_n^{s-j}\left(\epsilon_i^*\otimes\epsilon_j^*\right)\left(\hat{\Psi}_n^{s-i}\right)^T \right\rangle_\mathcal{S}
\end{aligned}
\end{equation*}
As $k\neq l$, we have $\E^* \epsilon_k'\otimes\epsilon_l'(z)=\E^* \left\langle\epsilon_k',z\right\rangle\epsilon_l'=\E^* \left\langle\epsilon_k',z\right\rangle\E^* \epsilon_l'=0$ for all $z$, i.e. $\E^* \epsilon_k'\otimes\epsilon_l' = 0$ and, analogously, $\E^*\epsilon_k^*\otimes \epsilon_l^*=0$. Moreover, if e.g. $j \neq k,l$, we have for arbitrary $X, y \in \mathcal{H}$
\begin{eqnarray*}
\E^* \left\langle\Psi^{t-l}\left(\epsilon_k'\otimes\epsilon_l'\right)(z),\Psi^{s-j}\left(\epsilon_i'\otimes\epsilon_j'\right)(y)\right\rangle &=& \E^* \left\langle\epsilon_k',z \right\rangle\left\langle\Psi^{t-l}\epsilon_l',\Psi^{s-j}\epsilon_j'\right\rangle\left\langle\epsilon_i',y\right\rangle \\
&=& \left\langle\E^* \left\lbrace\left\langle\epsilon_k',z \right\rangle\left\langle\epsilon_i',y\right\rangle\Psi^{t-l}\epsilon_l'\right\rbrace, \E^* \Psi^{s-j}\epsilon_j'\right\rangle =0
\end{eqnarray*}
as $\E\Psi^{s-j}\epsilon_j'=\Psi^{s-j}\left(\E^* \epsilon_j'\right) = 0$. Together with the definition of $\langle ., . \rangle_\mathcal{S}$, we get
$$
\E^* \left\langle\Psi^{t-l}\left(\epsilon_k'\otimes\epsilon_l'\right)(\Psi^{t-k})^T,\Psi^{s-j}\left(\epsilon_i'\otimes\epsilon_j'\right)(\Psi^{s-i})^T\right\rangle_\mathcal{S} = 0.
$$
Analogously, the expectations of the other terms are vanishing, such that for $k \neq l, i \neq j, \ \E^* B_{klij}^{(s,t)}=0$ except for $k=i\neq l=j$ or $k=j \neq l=i$.
To get the expectations of the remaining terms, we decompose
\begin{eqnarray*}
&&\Psi^{t-l}\left(\epsilon_k'\otimes\epsilon_l'\right)\left(\Psi^{t-k}\right)^T-\hat{\Psi}_n^{t-l}\left(\epsilon_k^*\otimes\epsilon_l^*\right)\left(\hat{\Psi}_n^{t-k}\right)^T \\
&&\hspace{6mm}=\left(\Psi^{t-l}-\hat{\Psi}_n^{t-l}\right)\left(\epsilon_k'\otimes\epsilon_l'\right)\left(\Psi^{t-k}\right)^T + \hat{\Psi}_n^{t-l}\left(\epsilon_k'\otimes\epsilon_l'\right)\left(\Psi^{t-k}-\hat{\Psi}_n^{t-k}\right)^T \\
&&\hspace{10mm} + \hat{\Psi}_n^{t-l}\left(\epsilon_k'\otimes\epsilon_l'-\epsilon_k^*\otimes\epsilon_l^*\right)\left(\hat{\Psi}_n^{t-k}\right)^T = \beta_{1,t}+\beta_{2,t}+\beta_{3,t}.
\end{eqnarray*}
Using $\left\Vert\Psi^j\right\Vert_{\mathcal{L}} \leq \left\Vert\Psi\right\Vert_{\mathcal{L}}^j < \hat{\delta}^j$ for some $\hat{\delta} <1$, (\ref{opnorm}) and (\ref{kronnorm}),
\begin{eqnarray*}
\left\Vert\beta_{1,t} \right\Vert_\mathcal{S} &\leq & \left\Vert\Psi^{t-l}-\hat{\Psi}_n^{t-l}\right\Vert_{\mathcal{L}}\left\Vert\epsilon_k'\otimes\epsilon_l'\right\Vert_{\mathcal{S}}\left\Vert\Psi^{t-k}\right\Vert_{\mathcal{L}}\\
&\leq & D\hat{\delta}^{2t-k-l}\left\Vert\Psi-\hat{\Psi}_n\right\Vert_{\mathcal{L}}\ \Vert\epsilon_k'\Vert\ \Vert\epsilon_l'\Vert
\end{eqnarray*}
for some generic constant $D$ from Lemma \ref{lemma4}.
Analogously,
\begin{eqnarray*}
\left\Vert\beta_{2,t} \right\Vert_\mathcal{S} &\leq & D\hat{\delta}^{2t-k-l}\left\Vert\Psi-\hat{\Psi}_n\right\Vert_{\mathcal{L}}\ \Vert\epsilon_k'\Vert\ \Vert\epsilon_l'\Vert\\
\left\Vert\beta_{3,t} \right\Vert_\mathcal{S} &\leq & \hat{\delta}^{2t-k-l}\left\Vert \epsilon_k'\otimes\epsilon_l'-\epsilon_k^*\otimes\epsilon_l^*\right\Vert_{\mathcal{S}}
\end{eqnarray*}
where we use $\left\Vert\hat{\Psi}_n^j\right\Vert_{\mathcal{L}} \leq \left\Vert\hat{\Psi}_n\right\Vert_{\mathcal{L}}^j$ and $\left\Vert\hat{\Psi}_n\right\Vert_{\mathcal{L}} \leq \hat{\delta}$ for large enough $n$ again from Lemma \ref{lemma4}. By, again, (\ref{kronnorm})
\begin{eqnarray}
\left\Vert\epsilon_k'\otimes\epsilon_l'-\epsilon_k^*\otimes\epsilon_l^*\right\Vert_{\mathcal{S}} &\leq& \left\Vert\left(\epsilon_k'-\epsilon_k^*\right)\otimes\epsilon_l'\right\Vert_{\mathcal{S}}+\left\Vert\epsilon_k^* \otimes \left(\epsilon_l'-\epsilon_l^*\right)\right\Vert_{\mathcal{S}} \nonumber \\
&=& \left\Vert\epsilon_k'-\epsilon_k^*\right\Vert \left\Vert\epsilon_l'\right\Vert + \left\Vert\epsilon_k^*\right\Vert\left\Vert\epsilon_l'-\epsilon_l^*\right\Vert \label{epsxeps}
\end{eqnarray}
such that
$$
\left\Vert\beta_{3t}(x)\right\Vert_\mathcal{S} \leq \hat{\delta}^{2t-k-l}\left\lbrace\left\Vert\epsilon_l'\right\Vert\left\Vert\epsilon_k'-\epsilon_k^*\right\Vert + \left\Vert\epsilon_k^*\right\Vert\left\Vert\epsilon_l'-\epsilon_l^*\right\Vert \right\rbrace .
$$
Now, as $k \neq l$,
\begin{eqnarray}
\left\vert\E^* B_{klkl}^{(s,t)}\right\vert &\leq&  \E^*\left\vert\left\langle\beta_{1t}+\beta_{2t}+\beta_{3t},\beta_{1s}+\beta_{2s}+\beta_{3s}\right\rangle_\mathcal{S} \right\vert \nonumber \\
&\leq & 4 D^2 \hat{\delta}^{2(t+s-k-l)}\left\Vert\Psi-\hat{\Psi}_n\right\Vert_{\mathcal{L}}^2\E^* \left\Vert\epsilon_k'\right\Vert^2\E^* \left\Vert\epsilon_l'\right\Vert^2 \nonumber \\
& & +4 D \hat{\delta}^{2(t+s-k-l)} \left\Vert \Psi-\hat{\Psi}_n \right\Vert_{\mathcal{L}} \E^* \left\Vert\epsilon_l'\right\Vert^2 \E^* \left(\left\Vert\epsilon_k'\right\Vert \left\Vert\epsilon_k'-\epsilon_k^*\right\Vert\right)  \nonumber \\
& &+ 4 D \hat{\delta}^{2(t+s-k-l)} \left\Vert \Psi-\hat{\Psi}_n \right\Vert_{\mathcal{L}}\E^* \left( \left\Vert\epsilon_k'\right\Vert \left\Vert\epsilon_k^*\right\Vert\right) \E^* \left( \left\Vert\epsilon_l'\right\Vert \left\Vert\epsilon_l'-\epsilon_l^*\right\Vert \right)  \nonumber \\
& & +\hat{\delta}^{2(t+s-k-l)}\E^* \left\lbrace\left\Vert\epsilon_l'\right\Vert\left\Vert\epsilon_k'-\epsilon_k^*\right\Vert + \left\Vert\epsilon_k^*\right\Vert\left\Vert\epsilon_l'-\epsilon_l^*\right\Vert\right\rbrace^2 \label{prTh3a}
\end{eqnarray}
Note that the expectation in the last line of (\ref{prTh3a}) may be written as 
\begin{eqnarray*}
&&\E^* \left\Vert\epsilon_l'\right\Vert^2\E^* \left\Vert\epsilon_k'-\epsilon_k^*\right\Vert^2 + 2\E^* \left(\left\Vert\epsilon_l'\right\Vert\left\Vert\epsilon_l'-\epsilon_l^*\right\Vert\right)\E^* \left(\left\Vert\epsilon_k^*\right\Vert\left\Vert\epsilon_k'-\epsilon_k^*\right\Vert\right) + \E^*\left\Vert\epsilon_k^*\right\Vert^2\E^*\left\Vert\epsilon_l'-\epsilon_l^*\right\Vert^2 \\
&& \hspace{6mm} \leq \left(\E^* \left\Vert\epsilon_l'\right\Vert^2 + 2\sqrt{\E^* \left\Vert\epsilon_l'\right\Vert^2}\sqrt{\E^*\left\Vert\epsilon_k^*\right\Vert^2} + \E^*\left\Vert\epsilon_k^*\right\Vert^2\right)d_2^2\left(F,\hat{F}_n\right) \\
&& \hspace{6mm} \leq 2\left(\E^* \left\Vert\epsilon_l'\right\Vert^2 + \E^*\left\Vert\epsilon_k^*\right\Vert^2\right)d_2^2\left(F,\hat{F}_n\right)
\end{eqnarray*}
due to our particular choice of $\left(\epsilon_k',\epsilon_k^*\right)$. Analogously, we get for the sum of the two terms involving expectations in the third and fourth line of (\ref{prTh3a})  that it is bounded by, using that $\epsilon_l',\epsilon_k'$ are identically distributed,
\begin{eqnarray*}
&& \E^* \left\Vert\epsilon_l'\right\Vert^2\sqrt{\E^* \left\Vert\epsilon_k'\right\Vert^2}d_2\left(F,\hat{F}_n\right) + \sqrt{\E^* \left\Vert\epsilon_k'\right\Vert^2}\sqrt{\E^*\left\Vert\epsilon_k^*\right\Vert^2}\sqrt{\E^* \left\Vert\epsilon_l'\right\Vert^2}d_2\left(F,\hat{F}_n\right) \\
&& \hspace{8mm} = \E^* \left\Vert\epsilon_k'\right\Vert^2\left(\sqrt{\E^* \left\Vert\epsilon_k'\right\Vert^2} + \sqrt{\E^*\left\Vert\epsilon_k^*\right\Vert^2}\right)d_2\left(F,\hat{F}_n\right)
\end{eqnarray*}
From Theorem \ref{theorem2}, we have $d_2^2\left(F,\hat{F}_n\right)=o_p(1)$ and, using Lemma 8.3 of \citep{Bickel}, $\E^*\left\Vert\epsilon_t^*\right\Vert^2 \underset{p}\rightarrow \E\left\Vert\epsilon_t\right\Vert^2$, i.e. $\E^*\left\Vert\epsilon_t^*\right\Vert^2=O_p(1)$. From Lemma \ref{lemma4}, $\left\Vert\Psi-\hat{\Psi}_n\right\Vert_{\mathcal{L}}\underset{a.s.}\rightarrow 0$. So, we have with some generic constant $D$
$$
\left\vert\E^* B_{klkl}^{(s,t)}\right\vert \leq D\hat{\delta}^{2(t+s-k-l)} o_p(1)
$$
uniformly in $k,l,s,t$. Analogously, we have the same upper bound for $\left\vert\E^* B_{kllk}^{(s,t)}\right\vert$ too. Finally, we conclude, using that $k=i, l=j$ or $k=j, l=i$ is only possible for $k,l \le \min(s,t)$,
\begin{eqnarray*}
\left\vert\sum_{s,t=0}^{n-1} \E^* \left\langle c_t,c_s\right\rangle_\mathcal{S} \right\vert &\leq & \sum_{s,t=0}^{n-1}\sum_{k \neq l=1}^{min(s,t)}\hat{\delta}^{2(t+s-k-l)} o_p(1) \\
&\le& 2 \sum_{0 \le s \le t \le n-1} \hat{\delta}^{2(t-s)}  \sum_{k, l=1}^{s} \hat{\delta}^{2(s-k) + 2(s-l)}  o_p(1) \\
&\le& \sum_{t=1}^n \frac{1}{(1-\hat{\delta}^2)^3}\ o_p(1) = o_p(n).
\end{eqnarray*}

\textbf{b}) As the next term, we consider
\begin{equation*}
\sum_{s,t=0}^{n-1} \E^* \left\langle d_t, d_s \right\rangle_\mathcal{S} = \sum_{s,t=0}^{n-1} \sum_{k=1}^t\sum_{l=1}^s\E^* B_{kl}^{(s,t)}
\end{equation*}
where  $B_{kl}^{(s,t)} = \langle C_k^{(t)}, C_l^{(s)}\rangle_\mathcal{S}$ and
$$
C_k^{(t)} = \Psi^{t-k}\left(\epsilon_k'\otimes\epsilon_k' - \E^* \epsilon_k'\otimes\epsilon_k'\right)\left(\Psi^{t-k}\right)^T - \hat{\Psi}_n^{t-k}\left(\epsilon_k^*\otimes\epsilon_k^* - \E^*\epsilon_k^*\otimes\epsilon_k^*\right)\left(\hat{\Psi}_n^{t-k}\right)^T .
$$
Due to the linearity of the operators involved, recalling that $\hat{\Psi}_n^{t-k}$ is fixed in the bootstrap world, we have $\E^* C^{(t)}_k =0$. Using the independence of $\left(\epsilon_k',\epsilon_k^*\right), \ \left(\epsilon_l',\epsilon_l^*\right)$, we conclude $\E^* B_{kl}^{(s,t)}=0$ for $k \neq l$. For the remaining case $k=l$, as in a), we decompose $C_k^{(t)} $ into 3 terms, where now
\begin{eqnarray*}
\beta_{1t} &=& \left(\Psi^{t-k}-\hat{\Psi}_n^{t-k}\right)\left(\epsilon_k'\otimes\epsilon_k' - \E^* \epsilon_k'\otimes\epsilon_k'\right)\left(\Psi^{t-k}\right)^T \\
\beta_{2t} &=& \hat{\Psi}_n^{t-k}\left(\epsilon_k'\otimes\epsilon_k' - \E^* \epsilon_k'\otimes\epsilon_k'\right)\left(\Psi^{t-k}-\hat{\Psi}_n^{t-k}\right) \\
\beta_{3t} &=& \hat{\Psi}_n^{t-k}\left(\epsilon_k'\otimes\epsilon_k' - \E^* \epsilon_k'\otimes\epsilon_k' - \epsilon_k^*\otimes\epsilon_k^* + \E^* \epsilon_k^*\otimes\epsilon_k^*\right)\left(\hat{\Psi}_n^{t-k}\right)^T
\end{eqnarray*}
such that
$$
B_{kk}^{(s,t)}= \left\langle\beta_{1t}+\beta_{2t}+\beta_{3t},\beta_{1s}+\beta_{2s}+\beta_{3s}\right\rangle_\mathcal{S} .
$$
For the first two terms, we have, using (\ref{kronnorm}),
\begin{equation*}
\left\Vert\epsilon_k'\otimes\epsilon_k' - \E^* \left(\epsilon_k'\otimes\epsilon_k'\right)\right\Vert_{\mathcal{S}} \leq  \left\Vert\epsilon_k'\otimes\epsilon_k'\right\Vert_{\mathcal{S}} +\E^* \left\Vert\epsilon_k'\otimes\epsilon_k'\right\Vert_{\mathcal{S}} 
= \left\Vert\epsilon_k'\right\Vert^2 +\E^* \left\Vert\epsilon_k'\right\Vert^2
\end{equation*}
and we conclude as in a),\\
$$
\left\Vert\beta_{it}\right\Vert_\mathcal{S} = O_p(1)\hat{\delta}^{2(t-k)}\left\Vert\Psi-\hat{\Psi}_n\right\Vert_{\mathcal{L}}, \ i=1,2.
$$
with $\E O_p(1)=O(1)$ uniformly in $k,t$. For the third term, we abbreviate $\Delta_k=\epsilon_k'\otimes\epsilon_k' - \epsilon_k^*\otimes\epsilon_k^*$ such that
$$
\left\Vert\beta_{3t}\right\Vert_\mathcal{S} \leq \hat{\delta}^{2(t-k)}\left\Vert\Delta_k-\E^* \Delta_k\right\Vert_{\mathcal{S}} .
$$
Using Cauchy-Schwarz and those bounds on $\left\Vert\beta_{it}\right\Vert_\mathcal{S}, i=1, 2, 3$, we have for some generic constant $D$
\begin{eqnarray*}
\left\vert\E^* B_{kk}^{(s,t)} \right\vert &\leq & \E^* \left\vert\left\langle\beta_{1t}+\beta_{2t}+\beta_{3t},\beta_{1s}+\beta_{2s}+\beta_{3s}\right\rangle_\mathcal{S} \right\vert \\
&\leq & 4D\hat{\delta}^{2(t+s-2k)}\left\Vert\Psi-\hat{\Psi}_n\right\Vert_{\mathcal{L}}^2 \\
&& + 4D \hat{\delta}^{2(t+s-2k)}\left\Vert\Psi-\hat{\Psi}_n\right\Vert_{\mathcal{L}}\E^* \left\Vert\Delta_k-\E^* \Delta_k\right\Vert_{\mathcal{S}} \\
&& + \hat{\delta}^{2(t+s-2k)}\E^* \left\Vert\Delta_k-\E^* \Delta_k\right\Vert_{\mathcal{S}}^2 .
\end{eqnarray*}
As $\left\Vert\Psi-\hat{\Psi}_n\right\Vert_{\mathcal{L}} \underset{a.s.}\rightarrow 0$, and as, from Lemma \ref{lemma7} below,
\begin{equation*}
\E^* \left\Vert\Delta_k-\E^* \Delta_k\right\Vert_{\mathcal{S}} \leq \sqrt{\E^* \left\Vert\Delta_k-\E^* \Delta_k\right\Vert_{\mathcal{S}}^2} \ \text{and} \ \E^* \left\Vert\Delta_k-\E^* \Delta_k\right\Vert_{\mathcal{S}}^2 \leq \E^* \left\Vert\Delta_k\right\Vert_{\mathcal{S}}^2 \underset{p}\rightarrow 0
\end{equation*}
uniformly in $k$ as $\left(\epsilon_k',\epsilon_k^*\right)$ are identically distributed, we have
\begin{equation*}
\left\vert\E^* B_{kk}^{(s,t)}\right\vert = \hat{\delta}^{2(t+s-2k)}o_p(1)
\end{equation*}
uniformly in $k,s,t$. Hence, as for $k=l$, we have $k \leq \min (s,t)$
$$
\left\vert\sum_{s,t=0}^{n-1} \E^* \left\langle d_t,d_s\right\rangle_\mathcal{S} \right\vert \leq  \sum_{s,t=0}^{n-1}\sum_{k=1}^{\min(s,t)}\hat{\delta}^{2(t+s-2k)}\ o_p(1) = o_p(n)
$$
as the threefold sum is $O(n)$ by the same calculations as at the end of part a).

\textbf{c}) We consider a third case in the supplement \ref{supp} and show in detail that
$$
\sum_{s,t=0}^{n-1}\E^* \left\langle a_t,b_s\right\rangle_\mathcal{S} = o_p(1),
$$
i.e. it is of even smaller order than the terms considered in \textbf{a)} and \textbf{b}). The other components of $\sum_{s,t=0}^{n-1}\E^* \left\langle A_t,A_s\right\rangle_\mathcal{S}$ can be shown to be of order at most $o_p(n)$ in the same manner, and we conclude, from (\ref{d2Gamma})
\begin{eqnarray*}
n d^2_2\left(\hat{\Gamma}_n-\Gamma, \hat{\Gamma}_n^*-\hat{\Gamma}_n \right)  = \left\Vert x \right\Vert^2 o_p(1) .
\end{eqnarray*}
\end{proof}

\begin{lemma}\label{lemma7}
Let $\left(\epsilon_t',\epsilon_t^*\right), \ t=1,\ldots,n$, be defined as in the proof of Theorem \ref{bootmean}. Then, under the assumptions of that theorem, for all $k \ge 1$
\begin{equation*}
\E^* \left\Vert\epsilon_k' \otimes \epsilon_k'-\epsilon_k^* \otimes \epsilon_k^*\right\Vert_{\mathcal{S}}^2 \underset{p}\rightarrow 0 \quad \text{for } n \to \infty
\end{equation*}
\end{lemma}
\begin{proof}
From (\ref{epsxeps}) with $k=l$
\begin{equation*}
\left\Vert\epsilon_k' \otimes \epsilon_k'-\epsilon_k^* \otimes \epsilon_k^*\right\Vert_{\mathcal{S}}^2 \leq \left\Vert\epsilon_k'-\epsilon_k^* \right\Vert^2\left(\left\Vert\epsilon_k'\right\vert+\left\Vert\epsilon_k^* \right\Vert\right)^2.
\end{equation*}
For $n \to \infty$, the right-hand side converges to 0 in probability as $\left\Vert\epsilon_k'-\epsilon_k^* \right\Vert\underset{p}\rightarrow 0$, which follows from $d_2^2\left(F,\hat{F}_n\right)=\E^* \left\Vert\epsilon_k'-\epsilon_k^* \right\Vert^2\underset{p}\rightarrow 0$. The lemma then follows from a dominated convergence argument where we specify a real random variable $W = \left\Vert\epsilon_k'\right\Vert + U$ with $\left\Vert\epsilon_k^* \right\Vert \leq U, \E^* U^4< \infty$. Then 
$$
\left\Vert\epsilon_k'-\epsilon_k^* \right\Vert^2 \left(\left\Vert\epsilon_k'\right\vert+\left\Vert\epsilon_k^* \right\Vert\right)^2 \leq \left(\left\Vert\epsilon_k'\right\vert+\left\Vert\epsilon_k^* \right\Vert\right)^4 \le W^4.
$$
Note that $\E^* \left\Vert\epsilon_k' \right\Vert^4=\E\left\Vert\epsilon_k\right\Vert^4 < \infty$ by assumption, and, therefore, $\E^* U^4< \infty$.

Recall that $\epsilon_k^*$ can be written as $\tilde{\epsilon}_J$ with $J$ being a Laplace variable in $\{1,\ldots,n\}$, i.e. $pr\left(J=k\right)=\dfrac{1}{n}, \ k=1,\ldots,n$.
Hence, 
\begin{equation*}
\epsilon_k^*=\tilde{\epsilon}_J=\hat{\epsilon}_J-\dfrac{1}{n}\sum_{k=1}^n\hat{\epsilon}_k=X_J-\hat{\Psi}_n\left(X_{J-1}\right)-\dfrac{1}{n}\sum_{k=1}^nX_k+\dfrac{1}{n}\sum_{k=1}^n\hat{\Psi}_n\left(X_{k-1}\right)
\end{equation*}
and using $\left\Vert\hat{\Psi}_n\right\Vert_\mathcal{L} \leq \hat{\delta}$ for large enough $n$ from Lemma \ref{lemma4}, we get
\begin{equation*}
\left\Vert\epsilon_k^*\right\Vert \leq \left\Vert X_J \right\Vert + \hat{\delta}\left\Vert X_{J-1} \right\Vert + \left\Vert \dfrac{1}{n}\sum_{k=1}^n X_k \right\Vert + \hat{\delta}\left\Vert \dfrac{1}{n}\sum_{k=1}^n X_{k-1} \right\Vert=U
\end{equation*}
We have $\E^* U^4<\infty$, as, e.g.,
\begin{equation*}
\E^*  \left\Vert X_J \right\Vert^4 = \dfrac{1}{n}\sum_{j=1}^n\left\Vert X_j \right\Vert^4 \leq C
\end{equation*}
for any $C>\E\left\Vert X_J \right\Vert^4$ and all large enough $n$ by the strong law of large numbers for strictly stationary real-valued time series.
\end{proof}

\begin{proof} (\textbf{Theorem \ref{bootacov}})\\
As in the proof of Theorem \ref{bootcov}, we get, using Lemma \ref{lemma5},
\begin{eqnarray*}
& &\left(\hat{C}_n'-C\right) -\left(\hat{C}_n^*-\hat{C}_n \hat{\Pi}_{k_n} \right) = \dfrac{1}{n} \sum_{t=0}^{n-1}A_t + O_p\left(\dfrac{1}{n}\right) \quad \quad \text{with now}\\
& &A_t = X_t'\otimes X_{t+1}'-\E^*\left(X_t'\otimes X_{t+1}'\right)-\left(X_t^*\otimes X_{t+1}^*-\E^*\left(X_t^*\otimes X_{t+1}^*\right)\right),
\end{eqnarray*}
From this point onwards, the proof follows exactly the same steps as the proof of Theorem \ref{bootcov} except that from the recursion (\ref{finMA}) and (\ref{kronecker}) we get an additional factor $\Psi$ resp. $\hat{\Psi}_n$ on the left hand side. E.g., we now have
$$
c_t = \sum_{k=1}^t \sum_{l=1, l \neq k}^{t+1} \left[\Psi^{t+1-l}\left(\epsilon_k'\otimes\epsilon_l'\right)\left(\Psi^{t-k}\right)^T- \hat{\Psi}_n^{t+1-l}\left(\epsilon_k^*\otimes\epsilon_l^*\right)\left(\hat{\Psi}_n^{t-k}\right)^T\right].
$$
As $\Vert \Psi \Vert_\mathcal{L}, \Vert \hat{\Psi}_n \Vert_\mathcal{L} < \hat{\delta} < 1$ a.s. for all large enough $n$, all the bounds of the proof of Theorem \ref{bootcov} remain valid.
\end{proof}

\section*{Acknowledgements}
The authors gratefully acknowledge support by the German Academic Exchange Service (DAAD) as well as by the Center for Mathematical and Computational Modelling (CM)$^\text{2}$ funded by the state of Rhineland-Palatinate.\\

\section{Supplementary Material: Details of Proofs} \label{supp}

\begin{proof} (\textbf{Lemma \ref{lemma1}})
\begin{eqnarray*}
\hat{\Gamma}_n\hat{\Gamma}_n^{\dagger}\left(x\right) &=& \hat{\Gamma}\left(\sum_{k=1}^{k_n} \dfrac{1}{\hat{\lambda}_k}\left\langle \hat{v}_k,x\right\rangle \hat{v}_k\right) \\
&=& \sum_{k=1}^{k_n}\dfrac{1}{\hat{\lambda}_k}\left\langle \hat{v}_k,x\right\rangle \hat{\Gamma}_n\left(\hat{v}_k\right) \\
&=& \sum_{k=1}^{k_n} \hat{v}_k \otimes \hat{v}_k\left(x\right)=\hat{\Pi}_{k_n}\left(x\right)
\end{eqnarray*}
as $\hat{\Gamma}_n\left(\hat{v}_k\right)=\hat{\lambda}_k\hat{v}_k$. Analogously, we get $\hat{\Gamma}_n^{\dagger}\hat{\Gamma}_n\left(x\right)=\hat{\Pi}_{k_n}\left(x\right)$. Now,
\begin{equation*}
\hat{\Psi}_n\left(x\right)=\hat{C}_n\hat{\Gamma}_n^{\dagger}\left(x\right)=\sum_{k=1}^{k_n}\dfrac{1}{\hat{\lambda}_k}\left\langle \hat{v}_k,x\right\rangle \hat{C}_n\left(\hat{v}_k\right)
\end{equation*}
\begin{eqnarray*}
\hat{\Psi}_n\hat{\Pi}_{k_n}\left(x\right) &=& \hat{\Psi}_n\left(\sum_{j=1}^{k_n}\left\langle \hat{v}_j,x \right\rangle \hat{v}_j\right) \\
&=& \sum_{j=1}^{k_n}\left\langle \hat{v}_j,x \right\rangle\hat{\Psi}_n\left(\hat{v}_j\right)\\
&=& \sum_{j=1}^{k_n}\left\langle \hat{v}_j,x \right\rangle \dfrac{1}{\hat{\lambda}_j}\hat{C}_n\left(\hat{v}_j\right)
\end{eqnarray*}
as $\hat{\Gamma}_n^{\dagger}\left(\hat{v}_j\right)=\dfrac{1}{\hat{\lambda}_j}\hat{v}_j$.
\end{proof}

\begin{proof} (\textbf{Lemma \ref{lemma2}})
From Lemma \ref{lemma1}, we have $\hat{\Psi}_n-\Psi\hat{\Pi}_p=\left(\hat{C}_n-\Psi\hat{\Gamma}_n\right)\hat{\Gamma}_n^{\dagger}$
\begin{eqnarray*}
n\left(\hat{C}_n-\Psi\hat{\Gamma}_n\right)\left(x\right)&=& \sum_{t=1}^n X_{t-1}\otimes X_t\left(x\right) - \Psi \left(\sum_{t=1}^n X_{t-1} \otimes X_{t-1}\left(x\right)\right) \\
&=& \sum_{t=1}^n \left\langle X_{t-1},x\right\rangle X_t - \sum_{t=1}^n\Psi \left(\left\langle X_{t-1},x\right\rangle X_t\right) \\
&=& \sum_{t=1}^n \left\langle X_{t-1},x\right\rangle \left(X_t-\Psi\left(X_{t-1}\right)\right)=\sum_{t=1}^n X_{t-1} \otimes \epsilon_t\left(x\right)
\end{eqnarray*}
\end{proof}

\begin{proof} (\textbf{Lemma \ref{lemma4} a)})
Note that \citep{Bosq} considers $\tilde{\Psi}_n=\hat{\Pi}_{k_n}\hat{\Psi}_n$ instead of $\hat{\Psi}_n$ as an estimate of $\Psi$. From the discussion in the proof of Theorem \ref{theorem2}, the conditions of Theorem 8.7 of \citep{Bosq} are satisfied. In our notation,
\begin{equation*}
\left\Vert\hat{\Psi}_n-\Psi\right\Vert_{\mathcal{L}} \leq \left\Vert\hat{\Psi}_n-\Psi\Pi_{k_n}\right\Vert_{\mathcal{L}} + \left\Vert\Psi\left(\Pi_{k_n}-\hat{\Pi}_{k_n}\right)\right\Vert_{\mathcal{L}} + \left\Vert\Psi\hat{\Pi}_{k_n}-\Psi\right\Vert_{\mathcal{L}}
\end{equation*}
From the proof of Theorem 8.7 of \citep{Bosq}, in particular (8.92), (8.93), the second and third terms converge to 0 a.s.\\

For the first term, we have in our notation for every $x$,
\begin{eqnarray*}
\left(\hat{\Psi}_n-\Psi\Pi_{k_n}\right)(x) &=& \hat{C}_n\left(\sum_{j=1}^{k_n}\dfrac{1}{\hat{\lambda}_j}\left\langle x,\hat{c}_j\hat{v}_j\right\rangle\hat{c}_j\hat{v}_j\right)-C\left(\sum_{j=1}^{k_n}\dfrac{1}{\lambda_j}\left\langle x,v_j\right\rangle v_j\right) \\
&=& \hat{C}_n\left(\sum_{j=1}^{k_n}\left(\dfrac{1}{\hat{\lambda}_j}-\dfrac{1}{\lambda_j}\right)\left\langle x,\hat{c}_j\hat{v}_j\right\rangle\hat{c}_j\hat{v}_j\right) \\
&& +\hat{C}_n\left(\sum_{j=1}^{k_n}\dfrac{1}{\lambda_j}\left(\left\langle x,\hat{c}_j\hat{v}_j\right\rangle -\left\langle x,v_j\right\rangle\right)\hat{c}_j\hat{v}_j\right) \\
&& +\hat{C}_n\left(\sum_{j=1}^{k_n}\dfrac{1}{\lambda_j}\left\langle x,v_j\right\rangle\left(\hat{c}_j\hat{v}_j-v_j\right)\right) \\
&& +\left(\hat{C}_n-C\right)\left(\sum_{j=1}^{k_n}\dfrac{1}{\lambda_j}\left\langle x,v_j\right\rangle v_j\right) \\
&=& a_{n1}(x)+a_{n2}(x)+a_{n3}(x)+a_{n4}(x).
\end{eqnarray*}
With $A_{ni}=\sup_{\left\Vert x\right\Vert \leq 1} a_{ni}(x), \ 1 \leq i \leq 4$, we have $\left\Vert\hat{\Psi}_n-\Psi\Pi_{k_n}\right\Vert_{\mathcal{L}} \leq \sum_{i=1}^4 A_{ni}$ and, from the proof of Theorem 8.7 of \citep{Bosq}, (8.84), (8.86), (8.88) and (8.90), we have $A_{ni} \underset{a.s.}\rightarrow 0$ for $i=1,\ldots, 4$.
\end{proof}

\begin{proof} (part of \textbf{Theorem \ref{bootcov}})
As another component of $\sum_{s,t=0}^{n-1}\E^* \left\langle A_t,A_s\right\rangle_\mathcal{S}$, we study here
\begin{equation*}
\sum_{s,t=0}^{n-1}\E^* \left\langle a_t,b_s\right\rangle_\mathcal{S} =\sum_{s,t=0}^{n-1}\sum_{k=1}^s\E^* B_k^{(s,t)}
\end{equation*}
where, with $\Xi_0=X_0'\otimes X_0' - \E^* X_0'\otimes X_0'$,
$$
B_k^{(s,t)} = \left\langle\Psi^t \Xi_0 \left(\Psi^t\right)^T - \hat{\Psi}_n^t \Xi_0 \left(\hat{\Psi}_n^t\right)^T, \Psi^{s-k}\left(X_0' \otimes \epsilon_k' \right)\left(\Psi^s\right)^T -\hat{\Psi}_n^{s-k}\left(X_0' \otimes \epsilon_k^* \right)\left(\hat{\Psi}_n^s\right)^T \right\rangle_\mathcal{S} .
$$
We decompose the left factor of the scalar product into $\gamma_{1t}+\gamma_{2t}$ with
\begin{eqnarray*}
\gamma_{1t}&=& \left(\Psi^t-\hat{\Psi}_n^t\right)\Xi_0\left(\Psi^t\right)^T\\
\gamma_{2t}&=& \left(\Psi^t\right)^T \Xi_0 \left(\Psi^t-\hat{\Psi}_n^t\right)^T
\end{eqnarray*}
Analogously, the second factor is $\beta_{1s}+\beta_{2s}+\beta_{3s}$ with
\begin{eqnarray*}
\beta_{1s} &=& \left(\Psi^{s-k}-\hat{\Psi}_n^{s-k}\right)\left(X_0' \otimes \epsilon_k' \right)\left(\Psi^s\right)^T\\
\beta_{2s} &=& \hat{\Psi}_n^{s-k}\left(X_0' \otimes \epsilon_k' \right)\left(\Psi^s-\hat{\Psi}_n^s\right)^T\\
\beta_{3s} &=& \hat{\Psi}_n^{s-k}\left(X_0' \otimes \epsilon_k' - X_0' \otimes \epsilon_k^* \right)\left(\hat{\Psi}_n^s\right)^T=\hat{\Psi}_n^{s-k}\left[X_0' \otimes \left(\epsilon_k'-\epsilon_k^*\right) \right]\left(\hat{\Psi}_n^s\right)^T
\end{eqnarray*}
As in part a) and b) of the proof, we have for some constant $D$
\begin{eqnarray*}
\left\Vert\gamma_{it}\right\Vert_\mathcal{S} &\leq & D\hat{\delta}^{2t}\left\Vert\Psi-\hat{\Psi}_n\right\Vert_{\mathcal{L}}\left\Vert \Xi_0 \right\Vert_{\mathcal{S}}, \ i=1,2 \\
\left\Vert\beta_{is}\right\Vert_\mathcal{S} &\leq & D\hat{\delta}^{2s-k}\left\Vert\Psi-\hat{\Psi}_n\right\Vert_{\mathcal{L}}\left\Vert X_0' \otimes \epsilon_k' \right\Vert_{\mathcal{S}}, \ i=1,2 \\
\left\Vert\beta_{3s}\right\Vert_\mathcal{S} &\leq &\hat{\delta}^{2s-k}\left\Vert X_0' \otimes \left(\epsilon_k'-\epsilon_k^*\right) \right\Vert_{\mathcal{S}} .
\end{eqnarray*}
We use 
\begin{eqnarray*}
&&\left\Vert\Xi_0\right\Vert_{\mathcal{S}} \leq \left\Vert X_0'\otimes X_0'\right\Vert_{\mathcal{S}} + \E^* \left\Vert X_0'\otimes X_0'\right\Vert_{\mathcal{S}} = \left\Vert X_0'\right\Vert^2 + \E^* \left\Vert X_0'\right\Vert^2 \\
&&\left\Vert X_0' \otimes \epsilon_k' \right\Vert_{\mathcal{S}} = \left\Vert X_0'\right\Vert \left\Vert \epsilon_k'\right\Vert \\
&&\left\Vert X_0' \otimes \left(\epsilon_k'-\epsilon_k^*\right) \right\Vert_{\mathcal{S}} = \left\Vert X_0'\right\Vert \left\Vert \epsilon_k'-\epsilon_k^*\right\Vert 
\end{eqnarray*}
Using Cauchy-Schwarz and independence of $X_0'$ and $\left(\epsilon_k',\epsilon_k^*\right)$, we have for some suitable constant $D$
\begin{eqnarray*}
\E^* \left\Vert\Xi_0\right\Vert_{\mathcal{S}}\left\Vert X_0' \otimes \epsilon_k' \right\Vert_{\mathcal{S}} &\leq & \E^* \left(\left\Vert X_0'\right\Vert^2 + \E^* \left\Vert X_0'\right\Vert^2\right)\left\Vert \epsilon_k'\right\Vert\left\Vert X_0'\right\Vert \\
&\le& 2 \E^* \left\Vert X_0'\right\Vert^3 \E^* \left\Vert \epsilon_k'\right\Vert < \infty \\
\E^* \left\Vert\Xi_0\right\Vert_{\mathcal{S}}\left\Vert\epsilon_k'-\epsilon_k^*\right\Vert\left\Vert X_0'\right\Vert &=& \E^* \left\Vert\epsilon_k'-\epsilon_k^*\right\Vert \E^* \left\Vert\Xi_0\right\Vert_{\mathcal{S}} \left\Vert X_0'\right\Vert \\
&\leq & d_2\left(F,\hat{F}_n\right)2 \E^* \left\Vert X_0'\right\Vert^3,
\end{eqnarray*}
such that
\begin{eqnarray*}
\left\vert\E^* B_k^{(s,t)}\right\vert &\leq & \E^* \left\vert\left\langle\gamma_{1t}+\gamma_{2t},\beta_{1s}+\beta_{2s}+\beta_{3s}\right\rangle_\mathcal{S} \right\vert \\
&\leq & 4D\hat{\delta}^{2(t+s)-k}\left\Vert\Psi-\hat{\Psi}_n\right\Vert_{\mathcal{L}}^2 \\
&& +2D\hat{\delta}^{2(t+s)-k}\left\Vert\Psi-\hat{\Psi}_n\right\Vert_{\mathcal{L}}d_2\left(F,\hat{F}_n\right) \\
&& = \hat{\delta}^{2(t+s)-k} o_p(1)
\end{eqnarray*}
uniformly in $t,s,k$. Therefore,
$$
\sum_{s,t=0}^{n-1}\E^* \left\langle a_t,b_s\right\rangle \leq  \sum_{s,t=0}^{n-1}\sum_{k=1}^s\hat{\delta}^{2(t+s)-k} o_p(1) 
\leq  \dfrac{1}{1-\hat{\delta}}\sum_{s,t=0}^{n-1}\hat{\delta}^{2t+s} o_p(1) 
\leq \dfrac{1}{(1-\hat{\delta})^2}\dfrac{1}{1-\hat{\delta}^2} o_p(1) .
$$
Hence, this term is of order $o_p(1)$.
\end{proof}

\bigskip

\bigskip
\textbf{Acknowledgement}: This paper was supported by the PhD programme \textit{Mathematics in Industry and Commerce} (MIC), funded by the German academic exchange service (DAAD), and by the \textit{Center for Cognitive Science}, Technische Universit{\"a}t Kaiserslautern\\

\bigskip

\bigskip
\noindent Corresponding author: \\
Prof. Dr. J{\"u}rgen Franke\\
Technische Universit{\"a}t Kaiserslautern, Department of Mathematics , Erwin-Schr\"odinger-Stra{\ss}e, D-67663 Kaiserslautern, Germany \\
Tel.: +49-(0)631-205-2741\\
Fax: +49-(0)631-205-2748\\
e-mail: franke@mathematik.uni-kl.de\\

{}

\end{document}